\documentclass[12pt,a4]{amsart}
\usepackage{amssymb,mathrsfs,amsmath,dsfont,array}

\usepackage{multirow,arydshln}
\usepackage{graphicx}
\usepackage{bm}
\usepackage{tikz}
\usepackage{capt-of}
\usepackage{mathbbol}     
\usepackage{color}
\usepackage{verbatim} 
\usepackage[all,cmtip]{xy}
\usepackage{amsbsy}
\usepackage{amsmath}
\usepackage{nccmath}

\usepackage{hyperref}
\usepackage{soul}

\newtheorem{Thm}{Theorem}[section]
\newtheorem{Pro}[Thm]{Proposition}

\newtheorem{lem}[Thm]{Lemma}

\newtheorem{deft}[Thm]{Definition}
\newtheorem{rem}[Thm]{Remark}

\makeatletter
\newcommand{\AlignFootnote}[1]{%
    \ifmeasuring@
    \else
        \footnote{#1}%
    \fi
}
\makeatother
\numberwithin{equation}{section}

\newcommand{\cA}{\mathcal{A}}

\renewcommand{\hat}{\widehat}
\renewcommand{\tilde}{\widetilde}

\def\bs{\boldsymbol}

\def\a1s{a_1,\cdots, a_s}
\def\a{\alpha}

\def\aa{\mathcal A}

\def\adot{\dot{\alpha}}
\def\bdot{\dot{\beta}}

\def\andd{\quad\hbox{and}\quad}

\def\b{\beta}

\def\bl4{B_{\ell\geq4}}

\def\bbbc{{\mathbb C}}

\def\d{\delta}
\def\D{\Delta}

\def\sce{\mathscr{E}}

\def\scf{\mathscr{F}}

\def\gG{\mathscr{G}}
\def\gG{{\mathcal G}}

\def\fg{\mathfrak{g}}

\def\scg{\mathscr{G}}

\def\sch{\mathscr{H}}

\def\hh{{\mathcal H}}

\def\fh{\mathfrak{h}}
\def\sch{\mathscr{H}}

\def\ii{\mathcal{I}}
\def\jj{\mathcal{J}}

\def\fk{\mathfrak{k}}

\def\lam{\lambda}

\def\LL{\mathcal{L}}
\def\scl{\mathscr{L}}

\def\ep{\epsilon}
\def\fm{(\cdot,\cdot)}

\def\bbbr{{\mathbb R}}

\def\supp{\hbox{\rm supp}}

\def\1k{\frac{1}{k}}
\def\op{\oplus}
\def\ot{\otimes}

\def\sub{\subseteq}
\def\sg{\sigma}

\def\pf{\noindent{\bf Proof. }}

\def\sspan{\hbox{\rm span}}

\def\ft{\mathfrak{t}}

\def\bbbz{{\mathbb Z}}

\def\1il{1\leq i\leq\ell}

\def\red{\textcolor[rgb]{1.00,0.00,0.00}}
\renewcommand{\hat}{\widehat}
\def\bs{\boldsymbol}
\newcommand{\Bigop}[2]{\raisebox{0.2ex}{\scalebox{0.7}{$\displaystyle \bigoplus_{#1}^{#2}\;$}}}
\newcommand{\summ}[1]{\raisebox{0.1ex}{\scalebox{1}{$\displaystyle \sum_{#1}\;$}}}

\begin{document}
\title{Quasi-integrable modules over affine Lie superalgebras (Critical level)}
\thanks{$^\star$ Corresponding author.}
\thanks{2020 Mathematics Subject Classification: 17B10, 17B67.}
\thanks{Key Words: Finite weight modules, 
	twisted affine Lie superalgebras}
\thanks{The third author's research was in part supported by a grant from IPM (No. 1403170414). 
This work is based upon research funded by Iran National Science Foundation INSF (No. 4001480).}
\maketitle
\pagestyle{myheadings}

\markboth{}{}
\centerline{Asghar Daneshvar$^{\rm a}$, Hajar Kiamehr$^{\rm a}$, Malihe Yousofzadeh$^{\rm a,b, \star}$}
	
\centerline{$^{\rm a}$ {\scalebox{0.65} {Department of Pure Mathematics, Faculty of Mathematics and Statistics, University of Isfahan,}}}
\centerline{{\scalebox{0.65} { P.O.Box 81746-73441, Isfahan, Iran,}
}}
\centerline{$^{\rm b}${\scalebox{0.65} { School of Mathematics
 Institute for Research in
	Fundamental Sciences (IPM), }}}
\centerline{{\scalebox{0.65} {P.O. Box: 19395-5746, Tehran, Iran.}
}}
\centerline{{\scalebox{0.65} {Email addresses: a.daneshvar@ipm.ir (Asghar Daneshvar), hkiamehr@sci.ui.ac.ir (Hajar Kiamehr),}}}	
	\centerline{{\scalebox{0.65} {		 ma.yousofzadeh@sci.ui.ac.ir \& ma.yousofzadeh@ipm.ir (Malihe Yousofzadeh).			
		}}}

\begin{abstract} 
Representation theory of Lie (super)algebras has attracted significant research interest  for many years,  especially due to  its applications in theoretical physics; in this regard, the representation theory of affine Lie (super)algebras is of central importance. To characterize simple modules over affine Lie (super)algebras, it is necessary to study the  cases of nonzero and critical levels separately.   Although a vast amount of research has been done on the representation theory of affine Lie (super)algebras $\LL$, investigations concerning  general modules at the critical level remain limited. In all existing studies, the characterization of the modules under consideration is reduced to the characterization of modules over some subalgebras of $\LL$. Depending on the structure of the original modules, these subalgebras -and the corresponding  modules- have different  natures some of which are already known, while others need to be studied separately. In this paper, we give a complete characterization of  the modules over specific subalgebras $\gG$ of a   twisted affine Lie superalgebra $\LL$  that arise in the study of general zero level  simple finite weight $\LL$-modules.   
In particular, in the special case that $\gG=\LL,$  we obtain a complete characterization of quasi-integrable $\LL$-modules of level zero. 
\end{abstract}

\section{Introduction}
The study of Lie (super)algebras  and their representations 
has attracted significant research attention, especially because of  their  applications in various domains of mathematics and physics; {(see for instance \cite{Freund} and \cite{KRW})}.  In this regard, affine Lie (super)algebras are important  tools  in number theory,  group theory, knot theory, and quantum field theory; see e.g., {\cite{CMY,Kac,KW2}.  Affine Lie (super)algebras were  introduced and classified by {V.~Kac and R.~Moody in non-super case} and by  J.W.~Van de Leur in the super case; according to their classification, affine Lie (super)algebras  are divided into untwisted and twisted types, see  \cite{Kac1, van-thes} (see also \cite[Appendix A]{you8}).

An affine Lie (super)algebra $\LL$ contains a Cartan subalgebra  $\hh$ with respect to which $\LL$ has a root space decomposition. The first step toward the classification of $\LL$-modules is dealing with  finite weight $\LL$-modules, i.e.,  $\LL$-modules having a weight space decomposition with respect to $\hh$ with finite weight  multiplicities.  An affine Lie (super)algebra contains a canonical central element; this central element acts  on a  simple finite {weight module $V$
as a scalar called the  \it level} of $V$. 

There has been a vast amount of research on finite weight modules over affine Lie (super)algebras (see \cite{you8} and the references therein); in this regard, the parabolic induction is a  powerful tool to characterize these modules: In general, a parabolic subset of the root system $\Phi$ of a Lie superalgebra $\fk$, having a root space decomposition, is a subset {$P\subseteq \Phi$}  satisfying 
\[{{\Phi}=P\cup -P\andd (P+P)\cap {\Phi}\sub P} .\] This gives the decomposition $\Phi=(P\setminus -P)\cup(P\cap-P)\cup (-P\setminus P)$ for $\Phi$ and respectively the decomposition 
$\fk=\fk^+_{_P}\op\fk^\circ_{_P}\op\fk^-_{_P}$ for $\fk.$ A parabolically induced $\fk$-module $V$ is the unique simple quotient of $U(\fk)\ot \Omega,$ in which $U$ stands for the universal enveloping algebra,  $\Omega$ is a simple $\fk^\circ_{_P}$-module considered as a module over $\fk^+_{_P}\op\fk^\circ_{_P}$ with trivial action of $\fk^+_{_P}$ and  the tensor product is taken  on the universal enveloping algebra of  $\fk^+_{_P}\op\fk^\circ_{_P}$; see \S~\ref{sect-2-1} for the details.  A simple $\fk$-module which is not parabolically induced  is called cuspidal. This in turn implies that the characterization of parabolically induced modules are reduced to the characterization of cuspidal modules.

To prove a  simple finite weight  $\fk$-module $V$  is parabolically induced,  one encounters three types of difficulties. The first one is finding a parabolic subset $P$  for which $\fk^+_{_P}$-invariants  $\Omega:=V^{\fk^+_{_P}}$ of $V$, namely, those  vectors of $V$  annihilated by the root vectors corresponding to $\fk^+_{_P}$, is nonzero. The second  is to ensure that the mentioned tensor product $U(\fk)\ot \Omega$ possesses a unique simple quotient. The most delicate issue, however,  is understanding the structure of $\Omega$ itself to allow the induction argument to proceed.

In \cite{Fer} and \cite{D-M-P}, the authors show that a finite weight module $V$ over  a finite dimensional  Lie superalgebra whose zero part is a  reductive Lie algebra is parabolically induced from  a cuspidal module; the parabolic subset getting involved in this setting  is based on the intersection of  the set of real roots whose corresponding root vectors acting locally nilpotently on $V$ with  the set of real roots  $\a$ for which nonzero  root vectors corresponding to $-\a$ act injectively on $V.$  In the affine case, the nature of the root system makes it more difficult  to identify the appropriate parabolic subsets required for such an induction. More precisely, the root system of an affine Lie superalgebra is decomposed into three parts: 1) real roots, 2) imaginary roots and 3) non-singular roots, which are respectively, 1) those roots  which are not self-orthogonal, 2) those  which are orthogonal to all  roots and 3) those roots which are self-orthogonal but not orthogonal to the entire root system. 
The set of imaginary roots is a free abelian group of rank 1, say, e.g., $\bbbz\d$ and there is a finite set $\dot R$ with a decomposition  $\dot R_0\cup \dot R_1,$ in which $\dot R_0$ is a finite root system, such that the root system $R$ of $\LL$ is a subset of $\dot R+\bbbz\d.$ For each $\dot \a\in \dot R,$ $R\cap (\dot\a+\bbbz\d)$ is of the form 
\begin{equation}\tag{$\dag$}\label{dag}
\dot\a+k_{\dot\a}\d+\bbbz r_{\dot\a}\d.
\end{equation}
for some $r_{\dot\a}\in\{1,2,4\}$ and $0\leq k_{\dot\a}<r_{\dot\a}.$ To study modules over an affine Lie superalgebra $\LL$, there exists a variety of issues that require to be addressed. Firstly, if the level of the module under consideration is zero or not; secondly, it is essential to understand how root vectors act on the module and interact with one another.  Regarding the second issue, one knows that real root vectors act on a  simple finite weight module either injectively or locally nilpotently which allows us to call  the  corresponding roots respectively  injective or locally nilpotent; these injective and locally nilpotent actions   are  well-behaved. More precisely,  if both injective and locally nilpotent roots appear in the  sequence (\ref{dag}) corresponding to a real root $\a,$ there exists a root $\gamma$ in this sequence such that all injective roots lie on one side of $\gamma$ while all locally nilpotent roots lie on the other side of $\gamma.$ We call such real roots $\a$ hybrid; if a real root is not hybrid, it means that either all roots in the mentioned sequence are locally-nilpotent or all are injective. We call such roots full-locally nilpotent and full-injective respectively. This, then, helps us to divide  the category  of finite weight $\LL$-modules into several subcategories. The nature of simple objects in each subcategory, offers an exclusive approach  to get a characterization of them.  
 
 In the cases  that either all real roots are hybrid or the level is nonzero, using a one-step parabolic induction, the characterization is reduced to the characterization of simple objects of a class of  well-studied finite weight modules over finite dimensional simple Lie superalgebras, but in the zero level, the situation is completely different; we  know from \cite{you13} that  in this case, we can reduce the characterization of $\LL$-modules  to the characterization   of modules $M$ over infinite dimensional subalgebras $\gG$ of $\LL,$ whose   real roots are either  full-locally nilpotent or hybrid.  Such modules are kind of generalization of $\LL$-modules called quasi-integrable modules; see Definition~\ref{quasi-} for the details.  
 
 In this paper, we study such $\gG$-modules and provide a complete characterization of them; in particular, we obtain a full characterization of zero level quasi-integrable 
$\LL$-modules. Our approach relies on a sequence of three successive parabolic inductions.
 The subalgebra $\gG$ is a sum of two subalgebras $\gG(1)$ and $\gG(2),$ moreover,  a $\gG$-module  $V$, as above,  is simultaneously an integrable $\gG(1)$-module\footnote{That is, all real roots are locally nilpotent.} and a hybrid $\gG(2)$-module\footnote{It means all real roots are hybrid.}.  Our first parabolic induction is obtained using integrability of $\gG(1)$-module $V.$  The most delicate step is the second one; at this stage, we encounter  a simple module $W$  over an extension of $\gG(2).$ This extension has a nonzero $\bbbz$-graded center. One of the most technical parts of our paper is proving that the nonzero homogenous elements  of the center, corresponding to the nonzero integer numbers, 
 act  trivially on $W$; this helps us to  define our second parabolic subset $P$ whose corresponding parabolic induction gives us a  simple finite weight module over a finite dimensional Lie superalgebra. Then, we use the results of \cite{D-M-P} and get our third parabolic induction which in fact completes the characterization.
 
 The layout of our paper is as follows: After this introduction, we have \S~2  which collects the necessary preliminaries for our main theorem and comprises three subsections: the first provides  some general information of the theory, the second addresses  some facts on representation theory of affine Lie algebras and the third one is exclusively devoted to the twisted affine Lie superalgebras, along with some technical facts we  need for our main theorem. We conclude the paper with \S~3 in which we prove our main theorem. 
\section{Preliminaries}\label{sect-2}
\subsection{Generic Information}\label{sect-2-1}
 Throughout the paper, unless otherwise mentioned, the underlying field, for all vector spaces, is the field of complex numbers $\bbbc$. Also, we denote  the dual space of a vector space $V$  by $V^*.$

Suppose that  $\fk=\fk_0\op\fk_1$ ($\bbbz_2=\{0,1\}$)  is  a Lie superalgebra equipped with an even  nondegenerate supersymmetric invariant  bilinear form $\fm$ and 
 $\fh$ is a finite dimensional subalgebra of the even part $\fk_0$ of $\fk$. A superspace  ${V}={V}_0\op {V}_1$ is called  an {\it {$\fh$}-weight  $\fk$-module} (or simply a weight $\fk$-module if there is no ambiguity) if
  \begin{itemize}
\item [(1)] $[x,y]v=x(yv)-(-1)^{\mid x\mid\mid y\mid}y(xv)$ for all $v\in {V}$ and  $x,y\in\fk_0\cup\fk_1$,
\item [(2)] $\fk_i {V}_j\sub {V}_{i+j }$ for all $i,j\in\{0,1\}$, 
 \item[(3)] ${V}=\op_{\lam\in \fh^\ast}{V}^\lam$ with
 ${V}^\lam:=\{v\in {V}\mid hv=\lam(h)v\;\;(h\in\fh)\}$ for each $\lam\in \fh^*.$  	
\end{itemize}

In this case, an element $\lam$ of 
$\supp({V}):=\{\lam\in \fh^*\mid {V}^\lam\neq \{0\}\},$ referred to as the {\it support} of $V,$ is called a {\it weight} of $V$, and the corresponding ${V}^\lam$ is called a {\it weight space}. Elements of a weight space are called {\it weight vectors}. If all weight spaces are  finite dimensional, then the module ${V}$ is called a {\it finite weight module}.

 If $\fk$ is an {$\fh$}-weight module via the adjoint representation, we say that $\fk$  has   a {\it root space decomposition} with respect to $\fh$; the set $\Phi$ of weights of $\fk$ is said to be the {\it root system} and weights, weight vectors and weight spaces are called   respectively {\it roots}, {\it root vectors} and  {\it root spaces}. We set 
 \[\Phi_i=\{\a\in \Phi\mid \fk^\a\cap\fk_i\neq \{0\}\}\quad (i=0,1).\]
From now until the end of this subsection, we   assume  $\fk$ has  a root space decomposition with respect to $\fh$ with the root system ${\Phi}$ such that   $\fk^0=\fh$. Therefore, the form on $\fh$ naturally 
 induces  a  nondegenerate symmetric bilinear form on the dual space $\fh^*$ of  $\fh$; we denote this new form   again by $\fm$.
For a subset $S$ of $\Phi$, define
 \begin{equation}  \label{decom-1}
 	\begin{array}{ll}
 		\hbox{\footnotesize $S_{re} := \{\alpha \in S \mid (\alpha,\alpha) \neq 0\},$}&
 		\hbox{\footnotesize $S_{im}: =\{\a\in S\mid(\a,\Phi)=\{0\}\}$},\\
 		\hbox{\footnotesize $S_{ns} :=\{\a\in S\setminus S_{im}\mid (\alpha,\alpha) = 0\},$}&
 		\hbox{\footnotesize $S^\times := S \setminus {S_{im}}.$}
 	\end{array}
 \end{equation}
 Elements of $S_{im}$ (resp. $S_{re}$ and $S_{ns}$) are called {\it imaginary roots} (resp. {\it real roots} and {\it nonsingular roots}).
 Next assume $V=\op_{\lam\in \fh^*}V^\lam$ is an   {$\fh$-weight $\fk$-module}  and $S\sub \Phi,$
 \begin{equation}\label{in-nil}
 	\parbox{4.9in}{we denote by  ${S}^{ln}({V})$ (resp. ${S}^{in}({V})$),   the set of all real roots $\a\in S$  for which $0\neq x\in\fk^\a$ acts on ${V}$ locally nilpotently (resp. injectively).}
 \end{equation}
If there is no ambiguity, we simply use  $S^\star$ in place of  $S^\star(V)$ for $\star  = in, ln$.
We set\footnote{We use $\#$ to indicate the cardinal number of a set.}
\begin{equation}\label{BC}
 	\begin{array}{l}
B_V:=\{\a\in \hbox{span}_\bbbz \Phi\mid  \#\{k\in\bbbz^{>0}\mid \lam+k\a\in \supp(V)\}<\infty \quad(\forall\lam\in\supp(V)\},\\
C_V:=\{\a\in \hbox{span}_\bbbz \Phi\mid   \a+\supp(V)\sub\supp(V)\}.
 	\end{array}
 \end{equation}
 We say that the  $\fh$-weight $\fk$-module $V$ has  {\it shadow} if
 \begin{itemize}
\item[\rm\bf (s1)]
${\Phi}_{re}=\{\a\in {\Phi}\mid (\a,\a)\neq 0\}={\Phi}^{in}(V)\cup {\Phi}^{ln}(V),$
 	
\item[\rm\bf (s2)] ${\Phi}^{ln}(V)=B_V\cap {\Phi}_{re}$ and ${\Phi}^{in}(V)=C_V\cap {\Phi}_{re}.$
 \end{itemize}

\subsubsection{Parabolically Induced Modules.}\label{par-subsets} For a subset $S$ of $\Phi$ and a functional  $\boldsymbol{f}:\hbox{span}_\bbbr S\longrightarrow \bbbr$,
 %
 the decomposition
 \begin{equation*}\label{tri-dec}
 	S=S^+\cup S^\circ\cup S^-
 \end{equation*}
 where
 $$\hbox{\small $S^{\pm}=S^{\pm}_{\bs{f}}:={\{\a\in S\mid \pm\boldsymbol{f}(\a)>0\}}\andd S^\circ=S^\circ_{\bs{f}}:=\{\a\in S\mid \boldsymbol{f}(\a)=0\}$}$$
 is said to be a {\it triangular decomposition} of $S$.
 It is called {\it trivial} if $S=S^\circ_{\bs{f}}.$

A subset $P$  of the root system $\Phi$  is called a {\it parabolic subset} of $\Phi$ if 
 $${\Phi}=P\cup -P\andd (P+P)\cap {\Phi}\sub P.$$ For a parabolic subset $P$ of $\Phi,$ we  have the decomposition
\[\fk=\fk^+_{_P}\op\fk^\circ_{_P}\op\fk^-_{_P}\] where
 $$\hbox{\small $\fk^\circ_{_P}:=\Bigop{\a\in P\cap -P}{}\fk^\a,\; \fk^+_{_P}:=\Bigop{\a\in P\setminus-P}{}\fk^\a\andd \fk^-_{_P}:=\Bigop{\a\in - P\setminus P}{}\fk^\a.$}$$
 Set   $\mathfrak{p}:=\fk^\circ_{_P}\op\fk^+_{_P}$. Each simple $\fk^\circ_{_P}$-module $N$ is a simple module over $\mathfrak{p}$ with the trivial action of $\fk^+.$ Then  $$\tilde N:=U(\fk)\ot_{U(\mathfrak{p})}N$$ is a $\fk$-module; here  $U(\fk)$ and $U(\mathfrak{p})$ denote, respectively,  the universal enveloping algebras of $\fk$ and $\mathfrak{p}.$ If the  $\fk$-module  $\tilde N$ contains a  unique maximal submodule $Z$ {intersecting} $N$ trivially, the quotient module $${\rm Ind}^{\fk}_{P}(N):=\tilde N/Z$$ is called  a {\it parabolically induced} module if $P\neq \Phi$. A simple  $\fk$-module  which is not parabolically induced is called {\it cuspidal}.

If  $\lam$ is a functional on the $\bbbr$-linear span of $\Phi,$  $P_\lam:=\Phi^+\cup \Phi^\circ$ is a parabolic subset of $\Phi.$ Moreover, if $\mu $ is a functional  on the $\bbbr$-linear span of $\Phi^\circ,$ we have a triangular decomposition $\Phi^\circ=\Phi^{\circ,+}\cup \Phi^{\circ,\circ}\cup \Phi^{\circ,-}$ for $\Phi^\circ$ and  $P_{\lam,\mu}:=\Phi^+\cup \Phi^{\circ,+}\cup \Phi^{\circ,\circ}$ is also a parabolic subset of $R.$ For $N$ as above, we  denote  ${\rm Ind}^{\fk}_{P_{\lam,\mu}}(N)$ simply by  
 ${\rm Ind}^{\fk}_{\lam,\mu}(N)$. We note that $P_{\lam,0}=P_{\lam}.$ In this case,  we denote $\fk_{_P}^\circ$ and $\fk_{_P}^\pm$  
 by $\fk_{_{\lam}}^\circ$ and $\fk_{_{\lam}}^\pm$  respectively and denote  ${\rm Ind}^{\fk}_{\lam,\mu}(N)$ by ${\rm Ind}^{\fk}_{\lam}(N)$. The following proposition is known in the literature; see e.g \cite[Proposition 3.3]{you8}: 
\begin{Pro}\label{ind}
 Suppose that $\lam$ and $\mu$ are as above and set $P:=P_{\lam,\mu}$. If $V$ is a simple  finite weight $\fk$-module with $$V^{\fk_{_P}^+}:=\{v\in V\mid \fk_{_P}^+ v=\{0\}\}\neq \{0\},$$ then $V^{\fk_{_P}^+}$ is a  simple finite weight $\fk_{_P}^\circ$-module and  $V\simeq {\rm Ind}_{\lam,\mu}^\fk(V^{\fk_{_P}^+}).$
\end{Pro}

\subsection{On split central extensions of affine Lie algebras} In this subsection, we present some results concerning integrable modules of affine Lie algebras that will be needed in the sequel.
 Throughout this subsection, assume 
 $\ft$ is an affine Lie algebra with Cartan subalgebra $\fh$ and   $Z$ is an abelian Lie algebra. For  \[\fk:=\ft\op Z,\] we assume that 
 \[\parbox{4.9in}{
 \it{ $V$ is a nonzero finite $(\fh\op Z=)\mathfrak H$-weight $\fk$-module on which the canonical central element of $\ft$ acts trivially and for all  real roots $\a$ of the root system $\phi$ of $\ft$ (and so of $\fk$), $x\in \fk^\a$ acts locally nilpotently on $V$.  }}\]
 One knows that there is an irreducible finite root system ${\dot\phi}$ and a free abelian group $\bbbz\d$ of rank $1$ such that the root system $\phi$ of $\fk$ is of one of the following forms:
\[\phi=\left\{\begin{array}{ll}
{\dot\phi}+\bbbz\d& X=Y^{(1)},\\
({\dot\phi}+2\bbbz\d)\cup(({\dot\phi}_{sh}\cup\{0\})+2\bbbz\d+\d)& X=A_{2\ell-1}^{(2)},D_\ell^{(2)},E_6^{(2)},\\
({\dot\phi}+3\bbbz\d)\cup(({\dot\phi}_{sh}\cup\{0\})+3\bbbz\d\pm\d)& X=D_4^{(3)},\\
({\dot\phi_{ind}}+2\bbbz\d)\cup({\dot\phi}+2\bbbz\d+\d)& X=A_{2\ell}^{(2)},
\end{array}
\right.
 \]
 {in which  $Y$ is the type of a finite dimensional simple Lie algebra}, \[{\dot\phi_{ind}}:={\dot\phi}\setminus {\dot\phi}_{ex}\] and  ``{\it sh}", ``{\it lg}" and ``{\it ex}"   stand for  ``short", ``long" and ``extra long" roots respectively. Set 
\[\fg:=\Bigop{\dot\a\in {\dot\phi_{{ind}}}}{}\fk^{\dot\a}.\]
Then, $\fg$ is a finite dimensional reductive  subalgebra of $\fk$ and $V$ is a finite weight $\fg$-module.
Using $\mathfrak{sl}_2$-module theory  and the fact that all root vectors corresponding to the real roots act locally nilpotently on $V$,  we can deduce that,
\begin{equation}\label{**}
\parbox{5.2in}{ for $\alpha \in \phi_{re}$ and $\lam\in \supp(V)$, $\frac{2(\lam,\a)}{(\a,\a)}\in\bbbz;$ moreover, if $\lam\in\supp(V)$ and  $\frac{2(\lam,\a)}{(\a,\a)}\in \bbbz^{>0}$ (resp. $\in \bbbz^{<0}$), then $\lam-\a\in \supp(V)$ (resp. $\lam+\a\in \supp(V)$).}
\end{equation}

 Fix a base $\D$ of ${\dot\phi_{ind}}$. The same argument  as stated in the proof of \cite[Lemma 2.6]{E2}  together with (\ref{**}),  implies that 
\begin{equation}\label{star}
\parbox{5.4in}{there is  $\lambda \in \supp(V)$ such that $\lambda +\adot\not\in\supp(V)$ for all $0\neq \adot\in \sspan_{\Bbb Z^{\geq 0}}\Delta,$  in particular, 
 $2(\lam,\adot)/(\adot,\adot)\in\bbbz^{\geq 0}$ for all $\adot\in {\dot\phi_{ind}}^+(\D):=({\dot\phi_{ind}}\setminus\{0\})\cap \sspan_{\Bbb Z^{\geq 0}}\Delta.$
}
\end{equation}

\begin{Thm}\label{affine-int}
{Let $\fk,\fg,\phi$ and $\Delta$ be as above}  and define a functional $$\bs{\zeta}: \sspan_\Bbb R \Delta(=\sspan_\bbbr\dot\phi) \longrightarrow \Bbb R$$
 such that 
$\bs{\zeta }(\Delta)\sub\bbbr^{>0}$.
Then, there is a nonzero weight vector $v$ such that
$\fk^{\adot+q\delta}v=\{0\}$ for all  $q\in \Bbb Z$
and $\adot \in \dot\phi$  with  $\bs{\zeta}(\adot)>0$.
\end{Thm}
\begin{proof}
Set  \[r:=\left\{\begin{array}{ll}
1&X=Y^{(1)},\\
3&X=D_4^{(3)},\\
2&\hbox{otherwise.}
\end{array}
\right.
\] 
Using (\ref{**}) and (\ref{star}) as well as  the same argument as in \cite[Theorem 2.4(ii)]{cheri},   we get that there is $\mu\in\supp(V)$ with 
	\begin{equation}\label{zero-1} {\mu+\dot\gamma+rq\d\not\in \supp(V) \andd 2(\mu,\dot\gamma)/(\dot\gamma,\dot\gamma)\in\bbbz^{\geq 0} \quad(\dot\gamma\in{\dot\phi_{ind}}^+(\D),q\in\bbbz).}\end{equation} In particular,
	\begin{equation}\label{zero}
	\fk^{\adot+q\delta}v=\{0\}\quad(v\in V^\mu,~\dot\a\in{\dot\phi_{ind}}^+(\D),~q\in r\bbbz).
	\end{equation}

	If there is a nonzero weight vector $v\in V^\mu$ with $\fk^{\dot\a+q\d} v=\{0\}$  for all  $q\in \Bbb Z$
and $\adot \in \dot\phi$  with  $\bs{\zeta}(\adot)>0$, then we are done. Otherwise, we pick $0\neq v\in V^\mu$
and note that  by (\ref{zero}), 
\[\cA:=\{\adot\in\dot\phi\mid\bs{\zeta}(\adot)>0, \fk^{\adot+p\delta} v\neq \{0\}\hbox{ for some $p\in\bbbz$}\}
\]	is not empty. 
Next pick  $\adot_*\in \cA$ with  $\bs\zeta(\adot_*)={\rm max}\{\bs\zeta(\adot)\mid \adot\in \cA\}.$  Since $\adot_*\in\cA,$ there is  $p_*\in \Bbb Z$ with
$$\fk^{\adot_*+ p_*\delta}v\neq \{0\}.$$  We let 
$0\neq w\in \fk^{\adot_*+p_*\delta}v$ and show that $\fk^{\adot+q\delta}w=\{0\}$ for all  $q\in \Bbb Z$
and $\adot \in \dot\phi$  with $\dot\a+q\d\in \phi$ and  $\bs{\zeta}(\adot)>0$.
We complete the proof through  the  following cases:

\smallskip

\noindent{$\bullet$ \underline{$\dot\a\in{\dot\phi_{ind}},~q\in r\bbbz$}.} In this case, using (\ref{zero}), we have 
\begin{align*}
\fk^{\adot+q\delta}w\sub\fk^{\adot+q\delta}\fk^{\adot_*+p_*\delta}v \subseteq 
[\fk^{\adot+q\delta},\fk^{\adot_*+p_*\delta}]v+
\fk^{\adot_*+p_*\delta}\underbrace{\fk^{\adot+q\delta}v}_{\{0\}}\sub
\fk^{\adot+\adot_*+(q+p_*\delta)\d}v.
\end{align*}
If $\adot+\adot_*+(q+p_*)\delta$ is not  a root, then  $\fk^{\adot+\adot_*+q\delta+p_*\delta}=\{0\}$ and there is nothing to prove. If $\adot+\adot_*+(q+p_*)\delta$ is a root, then  since 
$\zeta (\adot+\adot_*)>\zeta (\adot_*)$,  due to the choice of $\adot_*$, we have
 $\fk^{\adot+\adot_*+(q+p_*)\delta}v=\{0\}$ which in turn implies that $\fk^{\adot+q\delta}w=\{0\}$ and so we are done.

\smallskip
\noindent{$\bullet$ \underline{$X=A_{2\ell-1}^{(2)},D_\ell^{(2)},E_6^{(2)},D_4^{(3)}$, $\dot\a\in\dot\phi_{sh},~q\in r\bbbz\pm1$}.}
Note that as $\fk^{\adot_*+ p_*\delta}v\neq \{0\}$, by (\ref{zero}), we have $\adot_*\in\dot\phi_{sh}$ and   $p_*\in r\bbbz\pm1.$  In particular, we have 
\begin{equation}\label{pos-1}
\dot\a_*+(p_*+q)\d\in \phi\andd 2(\dot\a+\dot\a_*,\dot\a_*)/(\dot\a_*,\dot\a_*)>0.
\end{equation} On the other hand, we have
\begin{align*}
\fk^{\adot+q\delta}w\sub \fk^{\adot+q\delta}\fk^{\adot_*+p_*\delta}v\subseteq V^{\mu+(\adot+\adot_*)+(p_*+q)\delta}.
\end{align*}
If  $\fk^{\adot+q\delta}\fk^{\adot_*+p_*\delta}v\neq \{0\},$ we have $\mu+(\adot+\adot_*)+(p_*+q)\delta\in \supp(V).$ This together with (\ref{**}),  (\ref{zero-1}) 
and (\ref{pos-1})
implies that 
\[\mu+\adot=\mu+(\adot+\adot_*)+(p_*+q)\delta-(\adot_*+(p_*+q)\delta)\in\supp(V),\] which contradicts (\ref{zero-1}). Therefore, $\fk^{\adot+q\delta}\fk^{\adot_*+p_*\delta}v=\{0\}$ and so, we are done.

\smallskip
\noindent{$\bullet$ \underline{$X=A_{2\ell}^{(2)}$, $\dot\a\in\dot\phi,~q\in 2\bbbz+1$}.} In this case, we have $r=2.$
Since $\fk^{\adot_*+ p_*\delta}v\neq \{0\}$, we have    $p_*\in 2\bbbz+1.$  So \begin{equation}\label{new-added}\dot\phi_{ind}+(p_*+q)\d\sub \phi.\end{equation} 
We shall show $\fk^{\adot+q\delta}w=\{0\}.$ To this end, we note that 
\begin{align*}
\fk^{\adot+q\delta}w\sub \fk^{\adot+q\delta}\fk^{\adot_*+p_*\delta}v \subseteq V^{\mu+(\adot+\adot_*)+(p_*+q)\delta}.
\end{align*}
If  $\fk^{\adot+q\delta}\fk^{\adot_*+p_*\delta}v= \{0\},$ then, we are done. So, we assume $\fk^{\adot+q\delta}\fk^{\adot_*+p_*\delta}v\neq \{0\}$ and show that it leads to a contradiction.  We have $\mu+(\adot+\adot_*)+(p_*+q)\delta\in \supp(V).$ 
Hence,  we can contemplate the following cases:
\begin{itemize}
\item[(i)] If $\adot,\adot_*\in \dot\phi_{ind},$ then since either $2\frac{(\dot\a_*+\dot\a,\dot\a)}{(\dot\a,\dot\a)}\in\bbbz^{>0}$ or $2\frac{(\dot\a_*+\dot\a,\dot\a_*)}{(\dot\a_*,\dot\a_*)}\in\bbbz^{>0}$, we get using (\ref{**}), (\ref{zero-1}) and (\ref{new-added}) that $\mu+\adot\in\supp(V)$ or $\mu+\adot_*\in\supp(V)$ contradicting (\ref{zero-1}).
\item [(ii)]If exactly one of $\adot,\adot_*$ belongs to $\dot\phi_{ind},$ then, without loss of generality, we assume $\adot=2\dot\b$ for some short root $\dot\b$ and $\adot_*\in \dot\phi_{ind}.$ We  have $2\frac{(\mu+2\dot\b+\dot\a_*,\dot\b)}{(\dot\b,\dot\b)}>0$. Using (\ref{**}), this implies that $\mu+\dot\b+\dot\a_*+(p_*+q)\d\in\supp(V)$. We get a contradiction as in the previous case.
\item [(iii)] If $\adot,\adot_*\in \dot\phi_{ex},$ then, we have   $\dot\a=2\dot\b$ and $\dot\a_*=2\dot\b'$ for some
$\dot\b,\dot\b'\in {\dot\phi}_{sh}.$
In particular, we have 
\begin{equation*}\label{pos}
\dot\b'+(p_*+q)\d\in \phi\andd 2(\dot\b+\dot\b',\dot\b')/(\dot\b',\dot\b')>0.
\end{equation*} 
 This together with (\ref{**})
implies that 
\[\mu+2\dot\b+\dot\b'=\mu+2(\dot\b+\dot\b')+(p_*+q)\delta-(\dot\b'+(p_*+q)\delta)\in\supp(V).\] Similarly, we get 
$\mu+\dot\b+\dot\b'=\mu+2\dot\b+\dot\b'-\dot\b\in\supp(V),$ and finally, as $2(\mu+\dot\b+\dot\b',\dot\b')/(\dot\b',\dot\b')>0,$ 
 we have  $\mu+\dot\b=\mu+\dot\b+\dot\b'-\dot\b'\in\supp(V)$
which is a contradiction due to (\ref{zero-1}).  
\end{itemize}
The proof is now complete.
\end{proof}
 \subsection{Twisted affine Lie superalgebras}\label{affine}
Suppose that ${\mathscr{T}}$ is a basic classical simple Lie superalgebra of type $Y$ with standard Cartan subalgebra ${H}$ and corresponding root system $\dot {\mathfrak{s}}.$ Assume $\fm$ is an even  nondegenerate supersymmetric invariant bilinear form on ${\mathscr{T}}.$ One knows that if  $Y=D(k+1,\ell)$ with  $\ell\neq 0$ or 
$Y= A(k,\ell)$ with  $(k,\ell)\neq (1,1),(0,0)$, then,  there is an automorphism $\sg:{\mathscr{T}}\longrightarrow {\mathscr{T}}$ of order 
 \[l=\left\{\begin{array}{ll}
  4& Y=A(2m,2n),\\
 2&\hbox{otherwise},
 \end{array}
 \right.
\] such that $\sg({H})\sub {H}.$ The automorphism $\sg$ induces a linear automorphism on the dual space ${H}^*$ of ${H},$  mapping $\a\in{H}^*$ to $\a\circ \sg^{-1}.$  By the abuse of notations, we denote this new automorphism by $\sg$ as well. 
Considering the root space decomposition $\displaystyle{{\mathscr{T}}=\Bigop{\dot\a\in\dot{\mathfrak{s}}}{}{\mathscr{T}}^{\dot\a}}$ of ${\mathscr{T}}$ with respect to ${H},$ we get the weight space decomposition 
 ${\mathscr{T}}=\Bigop{\a\in \dot R}{}{\mathscr{T}}^{(\a)}$  of ${\mathscr{T}}$ with respect to the fixed point subalgebra $\fh$ of $H$  under $\sg$, in which 
 \[\dot R=\{\pi(\dot\a):=\dot\a|_{_{\fh}}\mid \dot\a\in \dot{\mathfrak{s}}\}
\andd {\mathscr{T}}^{(\pi(\dot\a))}=\sum_{\substack{\dot\b\in\dot{\mathfrak{s}}\\
\pi(\dot\a)=\pi(\dot\b)}}{\mathscr{T}}^{\dot\b}
\]
We mention  that 
\begin{equation}\label{special-case}
\parbox{3.25in}{
for all types other that type $X=A(2m,2n)^{(4)}$, we have $\pi(\dot\a)=0$ ($\dot\a\in\dot{\mathfrak{s}}$) if and only if $\dot\a=0.$ }
\end{equation}
Moreover, for $l$-th primitive root $\zeta$ of unity and 
${}^{[j]}{\mathscr{T}}:=\{x\in {\mathscr{T}}\mid \sg(x)=\zeta^jx\}$ for $j\in\bbbz,$ we have 
${\mathscr{T}}=\Bigop{j=0}{l-1}{}^{[j]}{\mathscr{T}}$ and that 
\[{}^{[j]}{\mathscr{T}}=\Bigop{\a\in \dot R}{}{}^{[j]}{\mathscr{T}}^{(\a)}\quad{\rm with } \quad{}^{[j]}{\mathscr{T}}^{(\a)}={}^{[j]}{\mathscr{T}}\cap {\mathscr{T}}^{(\a)} \quad (\a\in \dot R,~j\in \bbbz).\]
Assume $\bbbc c\op\bbbc d$ is a two dimensional vector space and set  \[\hat{\mathscr{T}}:=\Bigop{j=0}{l-1}({}^{[j]}{\mathscr{T}}\ot t^j\bbbc[t^{\pm l}]),\quad \LL:=\hat{\mathscr{T}}\op\bbbc c\op\bbbc d\andd 
 \LL_c:=\hat{\mathscr{T}}\op\bbbc c.\]
The 
superspace  $\LL$ together with the bracket 
$$[x\ot t^p+rc+sd,y\ot t^q+r'c+s'd]:=[x,y]\ot t^{p+q}+p(x,y)\d_{p+q,0}c+sqy\ot t^q-s'px\ot t^p$$
is a Lie superalgebra called {\it twisted affine Lie superalgebra} of type $Y^{(l)}$ (remember that $Y$ is the type of $\mathscr{T}$ and $l$ is the order of $\sg$). The subspace $\LL_c$ is an ideal of $\LL$  called the {\it core} of $\LL.$ The {\it centerless core} of $\LL$ is defined as
\begin{equation*}\label{centerless}
\LL_{cc}:=\LL_c/\bbbc c.
\end{equation*}
Set
\begin{equation}\label{cartan}
 \hh:=\fh\op\bbbc c \op\bbbc d\andd \sch:=\fh \op\bbbc d.
\end{equation}
The subspace $\hh$   is the standard Cartan subalgebra of $\LL$ and the root system of $\LL$ with respect to $\hh$ is 
\[R=\{\pi(\dot\a)+k\d\mid {}^{[k]}{\mathscr{T}}^{(\pi(\dot\a))}\neq \{0\}, k\in\bbbz,\dot\a\in\dot{\frak{s}}\}.\] 
For $\pi(\dot\a)+k\d\in R\setminus\{0\},$  we have $\LL^{\pi(\dot\a)+k\d}={}^{[k]}{\mathscr{T}}^{(\pi(\dot\a))}\ot t^k$ and $ \LL^{0}=\fh\op\bbbc c\op\bbbc d=\hh.$
Moreover, 
\[(x\ot t^p+rc+sd,y\ot t^q+r'c+s'd)=\d_{p+q,0}(x,y)+rs'+sr'\] defines an even nondegenerate supersymmetric invariant bilinear form on $\LL$ which is nondegenerate on $\hh$; in particular, it naturally induces a symmetric nondegenerate bilinear form on $\hh^*$ denoted again by $\fm.$  
Since the  form on  $\hh$ is a nondegenerate symmetric bilinear form, we conclude that
 \begin{equation}\label{tdota}
 \parbox{5.5in}{for each $\a\in \hh^*,$ there is a unique $t_\a\in\hh$ with $\a(h)=(t_\a,h)$ for all $h\in\hh.$}
 \end{equation}
We have 
$\hh= \bbbc c\op\bbbc d \op \sum_{\dot \a\in \dot R=\pi(\dot{\frak{s}})}\bbbc t_{\dot \a}$ with
\begin{equation}\label{talpha}(c,d)=1,(c,c)=(d,d)=0,(c,\sum_{\dot \a\in \dot R}\bbbc t_{\dot \a})=(d,\sum_{\dot \a\in \dot R}\bbbc t_{\dot \a})=\{0\}.\end{equation}
 We note that the superalgebra $\LL_{cc}\rtimes \bbbc d$ has a root space decomposition with respect to $\sch$ (see (\ref{cartan})) with the same root system $R$.
The root system $R$ of $\LL$ can be expressed as shown  in the following table:
\begin{table}[h]\caption{Root systems of twisted affine Lie superalgebras} \label{table1}
 {\footnotesize \begin{tabular}{|c|l|}
\hline
$\hbox{Type}$ &\hspace{3.25cm}$R$ \\
\hline
$\stackrel{A(2m,2n-1)^{(2)}}{{\hbox{\tiny$(m,n\in\bbbz^{\geq0},n\neq 0)$}}}$&$\begin{array}{rcl}
\bbbz\d
&\cup& \bbbz\d\pm\{\ep_i,\d_p,\ep_i\pm\ep_j,\d_p\pm\d_q,\ep_i\pm\d_p\mid i\neq j,p\neq q\}\\
&\cup& (2\bbbz+1)\d\pm\{2\ep_i\mid 1\leq i\leq m\}\\
&\cup& 2\bbbz\d\pm\{2\d_p\mid 1\leq p\leq n\}
\end{array}$\\
\hline
$\stackrel{A(2m-1,2n-1)^{(2)}}{{\hbox{\tiny$(m,n\in\bbbz^{>0},(m,n)\neq (1,1))$}}}$& $\begin{array}{rcl}
\bbbz\d&\cup& \bbbz\d\pm\{\ep_i\pm\ep_j,\d_p\pm\d_q,\d_p\pm\ep_i\mid i\neq j,p\neq q\}\\
&\cup& (2\bbbz+1)\d\pm\{2\ep_i\mid 1\leq i\leq m\}\\
&\cup& 2\bbbz\d\pm\{2\d_p\mid 1\leq p\leq n\}
\end{array}$\\
\hline
$\stackrel{A(2m,2n)^{(4)}}{ {\hbox{\tiny$(m,n\in\bbbz^{\geq0},(m,n)\neq (0,0))$}}}$& $\begin{array}{rcl}
\bbbz\d&\cup&  \bbbz\d\pm\{\ep_i,\d_p\mid 1\leq i\leq m,\;1\leq p\leq n\}\\
&\cup& 2\bbbz\d\pm\{\ep_i\pm\ep_j,\d_p\pm\d_q,\d_p\pm\ep_i\mid i\neq j,p\neq q\}\\
&\cup&(4\bbbz+2)\d\pm\{2\ep_i\mid 1\leq i\leq m\}\\
&\cup& 4\bbbz\d\pm\{2\d_p\mid 1\leq p\leq n\}
\end{array}$\\
\hline
$\stackrel{D(m+1,n)^{(2)}}{{\hbox{\tiny$(m,n\in\bbbz^{\geq 0},n\neq 0)$}}}$& $\begin{array}{rcl}
\bbbz\d&\cup&  \bbbz\d\pm\{\ep_i,\d_p\mid 1\leq i\leq m,\;1\leq p\leq n\}\\
&\cup& 2\bbbz\d\pm\{2\d_p,\ep_i\pm\ep_j,\d_p\pm\d_q,\d_p\pm\ep_i\mid i\neq j,p\neq q\}
\end{array}$\\
\hline
 \end{tabular}}
 \end{table}
 
Here  \[(\d,R)=\{0\},~ (\ep_i,\ep_j)=\d_{i,j},~(\d_p,\d_q)=-\d_{p,q}.\] Moreover, 
\[\dot R=\{\dot\a\in \sspan_\bbbz\{\ep_i,\d_p\mid i,p\}\mid (\dot\a+\bbbz\d)\cap R\neq \emptyset\}.\] For each $S\sub R,$ we define 
\begin{equation}\label{sdot}
\dot S:=\{\dot\a\in\dot R\mid S\cap (\dot\a+\bbbz\d)\neq \emptyset\}.
\end{equation}
We have 
$\dot R^\times=\dot R\setminus\{0\}=\dot R_{ns}\cup \dot R_{re}$ in which $\dot R_{ns}$ and $\dot R_{re} $ are as in the following table:

\begin{table}[h]\caption{Real and nonsingular roots of $\dot R$} \label{table-ns}
 {\footnotesize \begin{tabular}{|c|l|c|}
\hline
$\hbox{Type}$ &\hspace{2.5cm}$\dot R_{re}$ &$\dot R_{ns}$\\
\hline
$\stackrel{A(2m,2n-1)^{(2)}}{{\hbox{\tiny$(m,n\in\bbbz^{\geq0},n\neq 0)$}}}$&$\begin{array}{l}
\{\pm\ep_i,\pm\d_p,\pm2\ep_i,\pm2\d_p,
\ep_i\pm\ep_j,\d_p\pm\d_q\\
\mid 1\leq i\neq j\leq m,1\leq p\neq q\leq n\}\\
\end{array}$& $\begin{array}{l}
\{\pm \ep_i\pm\d_p\mid 1\leq i\leq m,1\leq p\leq  n\}\\
\vspace{-2mm}\\
\sub\dot R_{re}+\dot R_{re}
\end{array}
$\\
\hline
$\stackrel{A(2m-1,2n-1)^{(2)}}{{\hbox{\tiny$(m,n\in\bbbz^{>0},(m,n)\neq (1,1))$}}}$& $\begin{array}{l}
\{\pm2\ep_i,\pm2\d_p,
\ep_i\pm\ep_j,\d_p\pm\d_q\\
\mid 1\leq i\neq j\leq m,1\leq p\neq q\leq n\}\\
\end{array}$& $\begin{array}{l}
\{\pm\ep_i\pm\d_p\mid 1\leq i\leq m,1\leq p\leq  n\}\\
\vspace{-2mm}\\
\sub\frac{1}{2}(\dot R_{re}+\dot R_{re})
\end{array}
$\\
\hline
$\stackrel{A(2m,2n)^{(4)}}{ {\hbox{\tiny$(m,n\in\bbbz^{\geq0},(m,n)\neq (0,0))$}}}$& $\begin{array}{l}
\{\pm\ep_i,\pm\d_p,\pm2\ep_i,\pm2\d_p,
\ep_i\pm\ep_j,\d_p\pm\d_q\\
\mid 1\leq i\neq j\leq m,1\leq p\neq q\leq n\}\\
\end{array}$ & $\begin{array}{l} \{\pm\ep_i\pm\d_p\mid 1\leq i\leq m,1\leq p\leq  n\}
\\
\vspace{-2mm}\\
\sub\dot R_{re}+\dot R_{re}
\end{array}
$
\\
\hline
$\stackrel{D(m+1,n)^{(2)}}{{\hbox{\tiny$(m,n\in\bbbz^{\geq 0},n\neq 0)$}}}$& $\begin{array}{l}
\{\pm\ep_i,\pm\d_p,\pm2\d_p,
\ep_i\pm\ep_j,\d_p\pm\d_q\\
\mid 1\leq i\neq j\leq m,1\leq p\neq q\leq n\}\\
\end{array}$& $ \begin{array}{l}\{\pm\ep_i\pm\d_p\mid 1\leq i\leq m,1\leq p\leq  n \}
\\
\vspace{-2mm}\\
\sub\dot R_{re}+\dot R_{re}
\end{array}
$\\
\hline
 \end{tabular}}
 \end{table}

In  our main theorem,  we   will characterize elements of a class  of $\LL$-modules. For this, we need some information that we are gather in the following remark.  Due to the fact pointed out in (\ref{special-case}), we need to address separately cases $X=A(2m,2n)^{(4)}$ and $X\not=A(2m,2n)^{(4)}$ (the details  regarding this remark are available in \cite{KY}).
\begin{rem}\label{rem-MMM}
\rm  Assume $\scf\sub \LL_c$ is an $\hh$-invariant subalgebra of $\LL_c$, that is $[\scf,\scf]\sub \scf$ and  $\scf=\sum_{\a\in R}\scf^\a$ with $\scf^\a=\LL^\a\cap\scf.$ The quotient space  $\scf_{cc}:=(\scf+\bbbc c)/\bbbc c$ is a subalgebra of $\LL_{cc}$ which is identified with $\sum_{0\neq \a\in R}\scf^\a\op(\fh\cap \scf^0)$.  For now, we denote the bracket on $\LL_{cc}$ by $[\cdot,\cdot]_c.$
Now, suppose 
$S$ is a symmetric closed subset\footnote{It means that $S=-S$ and $(S+S)\cap R\sub S$.} of $R$ with $R\cap (S+\bbbz\d)\sub S$ and $2\bbbz\d\sub S_{im}$. For 
\[\scf(S):=\fh+\bbbc c+\sum_{0\neq \a\in S}\LL^\a+\sum_{k\in\bbbz}\LL^{(2k+1)\d},\] we have \begin{equation}\label{fscc}\scf(S)_{cc}:=\fh+\sum_{0\neq \a\in S}\LL^\a+\sum_{k\in\bbbz}\LL^{(2k+1)\d}.\end{equation}
Next assume  $$S_{im}=2\bbbz\d\andd (S+(2\bbbz+1)\d)\cap R=\emptyset.$$
\\
\noindent $\bullet$ $X\not =A(2m,2n)^{(4)}:$ In these cases, due to (\ref{special-case}), we have $[\sum_{k\in\bbbz}\LL^{(2k+1)\d},\sum_{k\in\bbbz}\LL^{(2k+1)\d}]_c=\{0\}$ and so we get $[\scf(S)_{cc},\sum_{k\in\bbbz}\LL^{(2k+1)\d}]_c=\{0\}.$
\smallskip
\\
\noindent $\bullet$ $X=A(2m,2n)^{(4)}:$ In this case 
\[\scf(S)_{cc}=\fh\op\sum_{ \a\in S^\times}\LL^\a\op\sum_{k\in\bbbz}\LL^{(2k+1)\d}\op\mathscr{K}\op\Bigop{k\in\bbbz}{}(\bbbc x_{4k+2}\op\bbbc y_{4k+2})\op((1-\d_{m,n})\ii)\ot t^2\bbbc[t^{\pm2}])\]
in which 
\begin{itemize}
\item $\ii$ is the diagonal block matrix in the basic classical simple Lie superalgebra $\mathscr{T}$ of type $A(2m,2n)$ whose first block diagonal is   $\frac{1}{2m+1}$ multiple of the identity matrix and the second block diagonal is   $\frac{1}{2n+1}$ multiple of the identity matrix.
\item $x_{4k+2},y_{4k+2}\in \LL^{(4k+2)\d}$ ($k\in \bbbz$)  are such that 
 \begin{equation}\label{y's} [\scf(S)_{cc},\Bigop{k\in\bbbz}{}\bbbc y_{4k+2}]_c=\{0\}\end{equation}
and \[\Bigop{k\in\bbbz}{}\LL^{(2k+1)\d}\op\Bigop{k\in\bbbz}{}(\bbbc x_{4k+2}\op\bbbc y_{4k+2})\] is a subalgebra  of $\LL_{cc}$ is isomorphic to the {\it quadratic} Lie superalgebra
\[\mathcal{Q}:=\underbrace{s^2\bbbc[s^{\pm4}]\op t^2\bbbc[t^{\pm4}]}_{\mathcal{Q}_0}\op \underbrace{t\bbbc[t^{\pm4}]\op t^{-1}\bbbc[t^{\pm4}]}_{\mathcal{Q}_1}\]
whose Lie bracket is 
\small{\begin{equation*}\label{bracket}\begin{array}{lll}
~[\mathcal{Q}_0,\mathcal{Q}_0]=[t^2\bbbc[t^{\pm4}], \mathcal{Q}_1]:=\{0\},& ~[t^{4k+1},t^{4k'-1}]:=0,&
~[t^{4k+1},t^{4k'+1}]:=t^{4(k+k')+2},\\
~[t^{4k-1},t^{4k'-1}]:=-t^{4(k+k')-2},&
~[s^{4k-2},t^{4k'+1}]:=t^{4(k+k')-1},&
~[s^{4k+2},t^{4k'-1}]:=t^{4(k+k')+1}.
\end{array}
\end{equation*}}
\item 
$\mathscr{K}:=\displaystyle{\sum_{\substack{
\dot\a\in\dot R\setminus \{0,\pm\ep_i,\pm\d_p\mid i,p\}\\k,k'\in\bbbz,k+k'\neq 0}}[\LL^{\dot\a+k\d},\LL^{-\dot\a+k'\d}]_c}
$ is an abelian subalgebra of $\scf(S)_{cc}\cap \Bigop{k\in\bbbz}{}\LL^{2k\d}$.\end{itemize}
Therefore, \begin{equation}\label{zfc}
\parbox{5.4in}{
for all types, whether $X=A(2m,2n)^{(4)}$ or $X\not =A(2m,2n)^{(4)},$
and  a symmetric closed subset $S$ of $R$  with $R\cap (S+\bbbz\d)\sub S$, $(S+(2\bbbz+1)\d)\cap R=\emptyset$ and  $S_{im}=2\bbbz\d$,  the center $Z(\scf(S)_{cc})$ of the subalgebra $\scf(S)_{cc}$ of $\LL_{cc}$ is nonzero. More precisely, there is $r\in \bbbz\setminus\{0\}$ such that $\LL^{r\d}\cap Z(\scf(S)_{cc})\neq \{0\}.$ }
\end{equation}
\end{rem}

\begin{Pro}\label{tetra}
Recall (\ref{sdot}) and (\ref{fscc}) and assume  $S$ is a symmetric closed subset of $R_0$ with $R\cap (S+\bbbz\d)\sub S$ and $2\bbbz\d\sub S_{im}$. Consider the subalgebra  $\scf(S)_{cc}$ of $\LL_
{cc}$ and suppose that   $V$ is a simple finite $\fh$-weight $\scf(S)_{cc}$-module.  For  $0\neq \dot\a\in \dot S$, the  following are equivalent:
\begin{itemize}
\item[(i)] If $\lam\in\supp(V)\sub\fh^*,$ then, $\lam+s\dot\a \not \in \supp(V)$ for all but finitely many $s\in \bbbz^{>0}.$
\item[(ii)] There exists $\lam\in\supp(V)$ such  that  $\lam+s\dot\a\not \in \supp(V)$ for all but finitely many $s\in \bbbz^{>0}.$
\item[(iii)] For all $k\in \bbbz$ with $\dot\a+k\d\in S,$  $0\neq x\in \LL^{\dot\a+k\d}$ acts locally nilpotently on $V$. 
\item[(iv)] There is $k\in \bbbz$ with $\dot\a+k\d\in S$ such that   $0\neq x\in \LL^{\dot\a+k\d}$ acts locally nilpotently on $V$. 
\end{itemize}
\end{Pro}
\begin{proof}  (i)$\Rightarrow$(ii) and  (iii)$\Rightarrow$(iv) are trivial. Also, (ii)$\Rightarrow$(iii) follows easily from the simplicity of $V$ together with the fact that $x\in \LL^{\dot\a+k\d}$ ($k\in \bbbz$) acts locally nilpotently on $\LL_{cc}.$  So, we just need to prove that (iv) implies (i). Assume to the contrary that (iv) holds but (i) is not satisfied. So, there is $\lam\in\supp(V)$ such that $\lam+s\dot\a\in\supp(V)$ for infinitely many $s\in\bbbz^{>0}.$
Assume $k$ and $x$ are as in statement of (iv).  Set $\a:=\dot\a+k\d.$  Pick $y\in\LL^{-\a}$ such that for $h:=[x,y]$ (modulo $\bbbc c$), $(x,y,h)$ is an $\mathfrak{sl}_2$-triple.  Set $G:=\bbbc x\op\bbbc y\op\bbbc h\simeq\frak{sl}_2$ which is a subalgebra of $\LL_{cc}$. Recalling (\ref{tdota}), we have  $U:=\sum_{s\in\bbbz}V^{\lam+s\dot\a}$ is a $G$-submodule of $V$ with the set of weights 
\[\left\{\frac{2\lam(t_{\dot\a})}{(\dot\a,\dot\a)}+2s\mid s\in \bbbz, \lam+s\dot\a\in \supp(V)\right\}\sub \bbbc.\] 
We notice  that $y$  acts injectively on $U$ as otherwise both $x$ and $y$ act locally nilpotently on $U$ and so  $U$ is a  completely reducible $G$-module with finite dimensional constituents; this in turn implies that  at least one of the weight spaces is infinite dimensional which is a contradiction.

For $s\in \bbbz^{>0}$, assume $\dot V(s)$ is the finitely generated $G$-submodule of $U$  generated by $V^{\lam+s\dot\a}.$ Each weight space of $\dot V(s)$ is invariant under the Casimir element $\mathfrak{c}=(h+1)^2+4yx$ of $G.$ Pick  an eigenvalue  $\tau_s$  of $\mathfrak{c}$ on $V^{\lam+s\dot\a}$ and set 
\begin{align*}
V(\tau_s):=\sum_{\mu\in \bbbc}\{v\in \dot V(s)^\mu\mid \exists r\in \bbbz^{>0}~\hbox{ s.t. } ~(\mathfrak{c}-\tau_s)^r(v)=0\}.
\end{align*}
We know that $V(\tau_s)$ has finite length, so,  there are a chain 
$V_1(\tau_s)\sub \cdots\sub V_{p_\tau}(\tau_s)=V(\tau_s)$
of submodules of $V(\tau_s)$ such that  $V_1(\tau_s)$ and $V_{i+1}(\tau_s)/V_i(\tau_s)$  ($1\leq i\leq p_\tau-1$) are simple $G$-modules. Since $x$ acts locally nilpotently on $U$, $x$ acts locally nilpotently on $V_1(\tau_s)$ and on $V_{i+1}(\tau_s)/V_i(\tau_s)$ for $1\leq i\leq p_\tau-1.$ In particular, 
by \cite[Proposition 3.55]{maz},  $V_1(\tau_s)$ as well as  $V_{i+1}(\tau_s)/V_i(\tau_s)$'s are highest weight modules of weights $\mu_1$ and $\mu_2$ respectively  with $(\mu_i+1)^2=\tau_s,$  for $i=1,2.$
So, $V(\tau_s)$ has a  highest weight of the form  $\mu_{\tau_s}= \frac{2\lam(t_{\dot\a})}{(\dot\a,\dot\a)}+2n_{\tau_s}$ for some $n_{\tau_s}\in\bbbz$ with $n_{\tau_s}\geq s$ and $(\mu_{\tau_s}+1)^2=\tau_s,$ This helps us to get positive integers 
\[s_1\leq n_{\tau_{s_1}}<s_2\leq n_{\tau_{s_2}}<\ldots \]
such that \[\frac{2\lam(t_{\dot\a})}{(\dot\a,\dot\a)}+n_{\tau_{s_1}}\not\in\bbbz^{<0},~~(\frac{2\lam(t_{\dot\a})}{(\dot\a,\dot\a)}+2n_{\tau_{s_i}}+1)^2=\tau_{s_i}\andd
 V(\tau_{s_i})^{\frac{2\lam(t_{\dot\a})}{(\dot\a,\dot\a)}+2n_{\tau_{s_i}}}\neq \{0\}\quad (i\in\bbbz^{\geq 1}).\]  It is not hard to see that $\tau_{s_i}$'s are distinct.  On the other hand, as $y$ acts injectively, 
we have \[\{0\}\neq y^r V(\tau_{s_i})^{\frac{2\lam(t_{\dot\a})}{(\dot\a,\dot\a)}+2n_{\tau_{s_i}}}\sub V(\tau_{s_i})^{\frac{2\lam(t_{\dot\a})}{(\dot\a,\dot\a)}+2n_{\tau_{s_i}}-2r}\quad(r\in\bbbz^{>0})\]
which  contradicts finite dimensionality of $V^\lam$. This completes the proof.
\end{proof}

\begin{Pro} \label{tight}
Assume $S$  is a symmetric closed subset of $R_0$ with $R\cap (S+\bbbz\d)\sub S$, $S_{re}\neq \emptyset$ and  $S_{im}=2\bbbz\d$ such that $(S+(2\bbbz+1)\d)\cap R=\emptyset.$ Consider 
$\scf(S)_{cc}$  as in (\ref{fscc}) and 
set \[\scg:=\scf(S)_{cc}\andd 
\scl:=\scf(S)_{cc}\rtimes\bbbc d.
\]
Recall (\ref{zfc}) and suppose that $M$ is a simple finite $\sch$-weight $\scl$-module  such that there is $z\in \LL^{r\d}\cap Z(\scg),$ for some nonzero $r\in\bbbz,$ which  acts nontrivially on $M$. Let $ \a\in S_{re}$. Then,  elements of $ \LL^{\a}$ act locally nilpotently on   $M$ if and only if  elements of  $\LL^{\a+s\d}$ act locally nilpotently on  $M,$ for all $s\in \bbbz$.
\end{Pro}
\begin{proof} Fix $\lam\in \supp(M).$ Since $M$ is a simple $\scl$-module, we have $\supp(M)\sub \lam+\fh^*+\bbbz\d.$ Assume $z$ is as in the statement. Since $M$ is a simple $\scl$-module and  $\{zv\mid v\in M\}$ as well as  $ \{v\in M\mid zv=0\}$ are $\scl$-submodules of $M$, we have $\{zv\mid v\in M\}=M$ and $ \{v\in M\mid zv=0\}=\{0\}.$ Therefore, 
the map 
\begin{align*}
f:&M\longrightarrow M\\
&v\mapsto zv\quad\quad(v\in M)
\end{align*} 
is a $\scg$-module  isomorphism, say with the inverse $g.$ Now,
\[W:=\{v-f(v)\mid v\in M\}\] is an $\fh$-weight  $\scg$-submodule of $M$ which is invariant under $g$  and so we have 
\[M=W+(\Bigop{\dot\a\in \fh^*}{}\Bigop{i=0}{|r|-1}M^{\lam+\dot\a+i\d}).\] The quotient module $M/W$ is a nonzero  finite $\fh$-weight  $\scg$-module with
$\supp(M/W)\sub \lam\mid_\fh+\fh^*$ 
and 
 $(M/W)^{\lam\mid_{\fh}+\dot\a}\sub\sum_{i=0}^{|r|-1}(M^{\lam+\dot\a+i\d}+W)/W$ for $\dot\a\in \fh^*$.
Since $M$ is a simple $\scl$-module, for a nonzero $\sch$-weight vector $v,$ we have 
\[M=U(\scl)v=U(\scg)U(\bbbc d)v=U(\scg)v,\] that is 
$M$ is cyclic $\scg$-module and so  $M/W$ is a cyclic $\scg$-module. Therefore, $M/W$ has a maximal submodule $U/W.$ So, 
\[\bar M:=M/U=(M/W)/(U/W)\] is a simple  finite  $\fh$-weight $\scg$-module. 
Suppose that $\bar{~~}: M\longrightarrow \bar M$ is the canonical projection map.  
We know that \[\scg=\Bigop{k\in\bbbz}{}\scg^k\hbox{ with } \scg^k=\left\{
\begin{array}{ll}
\fh& k=0,\\
\LL^{k\d}& k\in 2\bbbz+1.\\
\displaystyle{\sum_{\dot\a\in \dot S}\LL^{\dot\a+k\d}}&k\in2\bbbz\setminus\{0\}.
\end{array}
\right.
\]
Then, \[L(\bar M)=\bar M\ot \bbbc[t^{\pm1}]\] is a finite  $\sch$-weight  $\scl$-module via the action
 \begin{align*}
 \cdot:&\scl\times L(\bar M)\longrightarrow L(\bar M)\\
 &(x,\bar v\ot t^k)\mapsto \left\{
 \begin{array}{ll}
x\bar v\ot t^{s+k} & x\in\scg^{s} ~~~(s\in\bbbz), \\
 (\lam(d)+k) (\bar v\ot t^k) & x=d,
 \end{array}
 \right.
 \end{align*}
for  $v\in \bar M$ and $k\in\bbbz$.  
We have 
 \begin{align*}
 \psi: M\longrightarrow &\sum_{\dot\a\in \fh^*}\sum_{s\in \bbbz}(\overline{M^{\lam+\dot\a+s\d}}\ot t^{s})\sub L(\bar M),\\
 &v\mapsto  \bar v\ot t^{s}\quad\quad\quad\quad(v\in M^{\lam+\dot\a+s\d},~\dot\a\in \fh^*,~s\in \bbbz)
 \end{align*}
 is a nonzero  $\scl$-module  epimorphism and so an  $\scl$-module isomorphism as $M$ is simple.
This, together with  Proposition~\ref{tetra} implies that for $\a\in S_{re},$ $x\in \scl^{\a}$ acts locally nilpotently on $M$ if and only if $x\in \scl^{\a}$ acts locally nilpotently on $\bar M$ if and only if $x\in \scl^{\a+s\d}$ acts locally nilpotently on $\bar M$ for all $s\in \bbbz$ if and only if $x\in \scl^{\a+s\d}$ acts locally nilpotently on $M$ for all $s\in\bbbz.$
 \end{proof}
\subsubsection{Well-behaved interactions} 
Assume $\LL$ is a twisted affine Lie superalgebra with the standard Cartan subalgebra $\hh$ and corresponding root system $R$ as in Table~\ref{table1}.
 \begin{Thm}[{\cite[Theorem 4.8]{you8}}]\label{hybrid-0}
	\label{property}
Suppose that $V$ is an $\hh$-weight $\LL$-module having shadow. Then, for each $\b\in R_{re},$ one of the following will happen:
\begin{itemize}
\item[\rm (i)]  $(\b+\bbbz\d)\cap R\sub R^{ln}(V),$
\item[\rm (ii)] $(\b+\bbbz\d)\cap R\sub R^{in}(V),$
\item[\rm (iii)] there exist $r\in\bbbz$ and $t\in\{0,1,-1\}$ such that for $\gamma:=\b+r\d,$
\begin{align*}
&(\gamma+\bbbz^{\geq 1}\d)\cap R\sub R^{in}(V),\quad (\gamma+\bbbz^{\leq 0}\d)\cap R\sub R^{ln}\\
& (-\gamma+\bbbz^{\geq t}\d)\cap R  \sub R^{in}(V),\quad (-\gamma+\bbbz^{\leq t-1}\d)\cap R\sub R^{ln},
\end{align*}
\item[\rm (iv)] there exist $r\in\bbbz$ and $t\in\{0,1,-1\}$ such that for $\eta:=\b+r\d,$
\begin{align*}
&(\eta+\bbbz^{\leq -1}\d)\cap R\sub R^{in}(V),\quad (\eta+\bbbz^{\geq 0}\d)\cap R\sub R^{ln}(V),\\
& (-\eta+\bbbz^{\leq -t}\d)\cap R  \sub R^{in}(V),\quad (-\eta+\bbbz^{\geq 1-t}\d)\cap R\sub R^{ln}(V).
\end{align*}
\end{itemize}
\end{Thm}
{Assume $\dot\b\in\dot R_{re}$ and $\b\in R\cap (\dot\b+\bbbz\d).$ We call $\dot\b$ and $\b$ full-locally nilpotent,  full-injective,  down-nilpotent hybrid and up-nilpotent hybrid respectively if    for $\b$ properties stated in  (i) up to (iv)  occurs respectively.  
The set of full-locally nilpotent roots, full-injective roots and hybrid roots are denoted by $R_{\rm f-nil},$ $R_{\rm f-inj}$ and $R_{\rm hyb}$ respectively.}

{ A subset $S$ of $R$ is called   {\it full-locally nilpotent} (resp.  {\it full-injective},  {\it down-nilpotent hybrid}, {\it  up-nilpotent hybrid}, {\it hybrid}) if all elements of $S_{re}=S\cap R_{re}$ are full-locally nilpotent (resp. full-injective, down-nilpotent  hybrid, up-nilpotent  hybrid, hybrid).} 

\begin{deft}[{\cite[Definition~1]{you9}}]\label{quasi-}
{\rm Recall Table~\ref{table1} and set 

 \begin{equation}\label{r12}
 \begin{array}{l}
 \dot R(1):=\dot R\cap \sspan_\bbbz\{\ep_i\mid 1\leq i\leq m\},~~~  \dot R(2):=\dot R\cap \sspan_\bbbz\{\d_p\mid 1\leq p\leq n\} \andd \\
R(i):=(R^\times\cap (\dot R(i)+\bbbz\d))\cup\{k\d\in\bbbz\d\mid \exists \dot\a\in \dot R(i)^\times\hbox{ s.t. } \dot\a+k\d\in R\};~~(i=1,2).
 \end{array}
 \end{equation}
A simple finite $\hh$-weight  $\LL$-module is called {\it quasi-integrable} if for  some choice of $r,t$ with $\{1,2\}=\{r,t\},$
 \begin{itemize}
\item all real roots of $R(t)$ are full-locally nilpotent,
\item all real roots of $ R(r)$ are hybrid.
\end{itemize}
}
\end{deft}

  Using a    modification of the proof of \cite[Lemmas 5.6 \& 5.7]{you8}, one gets the following two lemmas:
    
\begin{lem}\label{lem:up and down}  {Assume $V$ is a weight $\LL$-module having shadow. Assume   $S$ is a symmetric closed subset of $R$ with  $S_{re}\neq \emptyset$ and ${(S+\bbbz\d)\cap R_0}\sub S.$ If
      $S$ is hybrid, then   either all real roots of $S$ are up-nilpotent hybrid or all are down-nilpotent hybrid.}
  \end{lem}
   \begin{lem}\label{lem:exist functional}
Keep the same notations and assumptions as in Lemma~\ref{lem:up and down}, and define
\begin{align*}
P:= \begin{cases}
S^{ln}(V)\cup -S^{in}(V)\cup  (\bbbz^{\geq0}\d\cap S)
& \mbox{if $S$ is up-nilpotent hybrid,}\\
S^{ln}(V)\cup -S^{in}(V)\cup 
(\mathbb{Z}^{\leq 0}\delta \cap S)& \mbox{if $S$ is down-nilpotent hybrid}.
\end{cases}\label{I}
\end{align*}
Then, there exists a functional $\boldsymbol{\zeta}: \text{\rm span}_{\mathbb{R}} ({S\setminus R_{ns}}) \rightarrow \mathbb{R}$ such that
  	$$P=\{\alpha \in S\setminus R_{ns}\; |\; \boldsymbol{\zeta}(\alpha) \geq 0\}.$$
  	In particular, $$\{\a\in S\cap R_{re}\; |\; \boldsymbol{\zeta}(\alpha)>0\}\subseteq S^{ln}(V) \quad \text{and}\quad \{\a\in S\cap R_{re} \;|\; \boldsymbol{\zeta}(\alpha)<0\}\subseteq S^{in}(V).$$
\end{lem}

 \begin{Pro}\label{1}
 	Suppose that  $S$ is a nonempty symmetric closed  subset of $R$ with $S_{re}=R_{re}\cap S\neq \emptyset$. Set \[S_0:= S \cap R_{0} ~ \hbox{ ~ as well as ~} ~ \dot{S}_0:=\{\dot{\alpha}\in {\dot R} \mid (\dot\a+\bbbz\d)\cap S_0\neq \emptyset\}\] and assume that \[(\dot{\alpha}+\mathbb{Z}\delta) \cap R_0 \subseteq S \quad\quad (\dot{\alpha} \in \dot S_0^\times=\dot S_0\setminus\{0\}).\]
Then the following statements hold:
\begin{itemize}
\item[(a)]	$\dot{S}_0$ is a finite root system, say e.g. with irreducible components  $\dot{S}(1),\ldots,\dot{S}({s}).$
\item[(b)]	  For $1\leq i\leq s,$ set  
$S{(i)}:={(\dot{S}{(i)}}+\mathbb{Z}\delta)\cap R_0$ and recall (\ref{decom-1}).
Then \[\mathcal{G}{(i)}:=\hh+\summ{\alpha \in S{(i)}^\times}\mathcal{L}^{\alpha}+\summ{\alpha,\b \in S{(i)}^\times}[\LL^\a,\LL^{\b}]\] is an affine Lie algebra, up to a central space. 	
\item[(c)]  Assume $V$ is a nonzero finite weight module over \[\mathcal{G}_{_S}:=\sum_{\alpha \in S}\mathcal{L}^{\alpha}\] such that $c$ acts trivially on $V$ and 
$S^{ln}(V)=S_{re}.$ 
Fix a base $\D$ of the finite root system  $\dot S_0$.  Assume $\bs{\zeta}$ is a functional on $\sspan_\bbbr \dot S_0$ with $\bs{\zeta}(\D)\sub \bbbr^{>0}$. Then, there exist $\mu \in \supp(V)$ and  $0\neq v \in V^{\mu}$ such that 
\[\LL^{\dot \a+r\d} v=\{0\}\quad\quad (r\in\bbbz,~\dot\a\in \dot S_0,~\bs{\zeta}(\dot\a)>0).\]
 \end{itemize}
 \end{Pro}
\begin{proof}
		(a) is easily verified; just note that as $2(R_1\cap R_{re})\sub R_0, $ we can deduce that $S_{0}\cap R_{re}\neq\{0\}$ and so $\dot S_0\neq \{0\}.$

(b) Recall (\ref{talpha}). We know  that the form $\fm$ is nondegenerate on $\bbbc c\op \bbbc d$ as well as on $\hh$ and on $\summ{\dot\a\in \dot S(i)^\times}\bbbc t_{\dot\a}$. Also, we have $(\bbbc c\op \bbbc d,\summ{\dot\a\in \dot S(i)^\times}\bbbc t_{\dot\a})=\{0\}$.  Setting 
\[\fh(i):=(\bbbc c\op \bbbc d)+\sum_{\a\in S(i)}\bbbc t_\a=\bbbc c\op \bbbc d\op\sum_{\dot\a\in  \dot S(i)^\times}\bbbc t_{\dot\a},\] we get that the form is nondegenerate on $\fh(i)$ and so, there is an orthogonal complement $\fh'(i)$ for $\fh(i)$ in $\hh;$ in particular,   for $\a\in S(i)$, $x\in \LL^\a$ and $h\in \fh'(i),$ we have 
\begin{align*}
[h,x]=\a(h)x=(t_\a,h)x\in(\fh(i),\fh'(i))x=\{0\},
\end{align*}
that is, $\fh'(i)$ is contained in the center of $\gG(i).$ We claim that 
\[\fk(i):=\fh(i)+\summ{\alpha \in S{(i)}^\times}\mathcal{L}^{\alpha}+\summ{\alpha,\b \in S{(i)}^\times}[\LL^\a,\LL^{\b}]\] is an affine Lie algebra. 

To this end, we first  note that  for each element $\a\in S(i),$ we have $\a(\fh'(i))=\{0\}$. So, for distinct elements  $\a,\b\in S(i),$ the restriction  of $\a$ to $\fh(i)$ is different from the restriction of  $\b$ to $\fh(i).$ This, in particular, implies that the decomposition  
\[\fk(i)=\fh(i)\op\Bigop{\alpha \in S{(i)}^\times}{}\mathcal{L}^{\alpha}\op\Bigop{0\neq k\in\bbbz}{}(\sum_{ \dot\a\in \dot S(i)^\times}\sum_{t\in\bbbz}[\LL^{\dot\a+t\d},\LL^{-\dot\a+(k-t)\d}])\] coincides with  the root space decomposition of $\fk(i)$ with respect to $\fh(i).$ It is now easily verified that  for 
\[(\fk(i))_c:=\sum_{ \a\in {S(i)^\times}}\LL^\a\op\sum_{ \dot\a\in \dot S(i)^\times}\sum_{k,k'\in\bbbz}[\LL^{\dot\a+k\d},\LL^{-\dot\a+k'\d}],\] we have 
\[\fk(i)=(\fk(i))_c\op\bbbc d\] and that $\{x\in \fk(i)\mid (x,(\fk(i))_c)=\{0\}\}\sub (\fk(i))_c.$
This together with  \cite[Theorem 2.32]{pin} implies that $\fk(i)$ is an affine Lie algebra, as we desired.

(c) By part~(b), for each $1\leq i\leq s$,  $\mathcal{G}{(i)}$ is an affine  subalgebra of $\mathcal{G}_{_S},$ up to a central subspace.	Set $V(0):=V$ and for $1\leq i\leq s,$ define
	\[V(i):=\text{span}_\bbbc\{v\in V(i-1) \mid \LL^{\dot\a+k\d} v=\{0\}~ (\dot\a\in \dot S(i),~\bs{\zeta}(\dot\a)>0, ~k\in\bbbz)\}.\]  
It is easy to see that  for $0\leq i\leq s-1,$ $V(i)$ is a finite weight $\gG(i+1)$-module. 
	 In particular, using Theorem~\ref{affine-int} together with an induction process, we get   $V(i+1)\neq \{0\}.$ In particular $V(s)\neq\{0\}.$
	This completes the proof.
	\end{proof}

\section{Main results}
Throughout this section, we follow the notations of Subsection~\ref{affine}; specially,  we 
assume $\LL$  is a twisted  affine Lie superalgebra of type $X$ with $\LL_1\neq \{0\},$ standard Cartan subalgebra $\hh$ and corresponding root system $R=R_0\cup R_1$ with $R\sub\dot R+\bbbz\d$; see Tables~\ref{table1}~\& \ref{table-ns} and (\ref{sdot}). We also mention that using Table \ref{table1}  and  \cite[Tables~ 2 and 5]{you10}, the following hold:
 \begin{equation}\label{exp}
\parbox{5.7in}{ \begin{itemize}
\item[(1)] For each $0\neq \dot\a\in\dot R,$ there are $r_{\dot\a}\in\{1,2,4\}$ and $0\leq k_{\dot\a}<r_{\dot\a}$ such that 
$(\dot\a+\bbbz\d)\cap R=\dot\a+k_{\dot\a}\d+r_{\dot\a}\bbbz\d$ and if $k_{\dot\a}\neq 0,$ then $k_{\dot\a}| r_{\dot\a}$. 
\item[(2)]  If  $\dot\a\in \dot R_{re}$, then  $(\dot\a+\bbbz\d)\cap R\cap R_1$ is one of the following: $\emptyset$, $\dot\a+\bbbz\d$, $\dot\a+2\bbbz\d$, 
$\dot\a+(2\bbbz+1)\d$. 
 \end{itemize}}  
\end{equation}

 \begin{lem}\label{M--30}
 Suppose $V$ is a zero level  simple finite weight  $\LL$-module. Assume $\varrho\in\{1,2\}$ and  $\dot\a,\dot\b\in \frac{1}{\varrho}\dot R$ with $(\dot\a,\dot\b)=0.$ If $\pm\varrho\dot\a$ are full-locally nilpotent  and $\pm\varrho\dot\b$ are either hybrid or full-locally nilpotent,  then, there are no $\mu\in\supp(V)$  and $m_s\in\bbbz$ $(s\in \bbbz^{>0})$ such that  $\mu+s(\dot\a+\dot\b)+m_s\d\in \supp(V)$ for all $s\in\bbbz^{>0}.$
  \end{lem}
  \begin{proof}
 Before starting the proof, we mention that  throughout the proof,  we frequently use the following fact which is known  from $\frak{sl}_2$ and $\frak{osp}(1,2)$-module theory (see \cite{you9}).
\begin{equation}\label{in-supp 1}
\parbox{6in}{If $\mu\in\supp(V)$ and $\pm\a\in R^{ln}(V)$  with $2(\mu,\a)/(\a,\a)>0,$ then $\mu-\a\in \supp(V).$}
\end{equation}
 
We prove the lemma  by contradiction. We assume the negation of the statement and drive a contradiction  using  the fact that  simple modules have shadow in particular $R^{ln}=B_V$; see (\ref{BC}).
So, to the contrary, assume there are $\mu\in\supp(V)$ and $m_s\in\bbbz$ ($s\in\bbbz^{>0}$) such that  $\mu+s(\dot\a+\dot\b)+m_s\d\in \supp(V)$  for all $s\in\bbbz^{>0}.$  In particular, we have \[\mu_{_{\varrho s}}:=\mu+s\varrho(\dot\a+\dot\b)+m_{\varrho s}\d\in \supp(V)\quad (s\in\bbbz^{>0}).\]
  We note that $(\dot\a,\dot\b)=0$ and  pick $s_0$ such that 
  \begin{equation}\label{s_0}
  \begin{split}
   & \frac{2(\mu_{{\varrho s_0}},\varrho \dot\a)}{(\varrho \dot\a,\varrho \dot\a)}=\frac{2(\mu+ s_0\varrho(\dot\a+\dot\b),\varrho \dot\a)}{(\varrho \dot\a,\varrho \dot\a)}=\frac{2(\mu,\varrho \dot\a)}{(\varrho \dot\a,\varrho \dot\a)}+2s_0\in\bbbz^{>0}
  \andd \\
&\frac{2(\mu_{{\varrho s_0}},\varrho \dot\b)}{(\varrho \dot\b,\varrho \dot\b)}=\frac{2(\mu+ s_0\varrho(\dot\a+\dot\b),\varrho \dot\b)}{(\varrho \dot\b,\varrho \dot\b)}
=\frac{2(\mu,\varrho \dot\b)}{(\varrho \dot\b,\varrho \dot\b)}+2s_0\in\bbbz^{>0}.   
\end{split}
  \end{equation}
We have 
\[\upsilon_s:=\mu_{{\varrho s_0+\varrho s}}=\mu_{\varrho s_0}+s\varrho(\dot\a+\dot\b)+(m_{\varrho s+\varrho s_0}-m_{\varrho s_0})\d\in \supp(V)\quad (s\in\bbbz^{>0}).\]

By our assumption, either $(\pm\varrho\dot\b+s\d)\cap R\sub R^{ln}$ for large enough $s$ or $(\pm\varrho\dot\b-s\d)\cap R\sub R^{ln}$ for large enough $s$.  The proofs of  the both cases are similar, so, we just carry out the proof for  the former case. Recalling $r_{\varrho\dot\b}$ and $k_{\varrho\dot\b}$ from (\ref{exp}), for a large enough  $m'\in\bbbz^{>0},$ we have    \[\left\{\begin{array}{ll}
   \varrho\dot\b+k_{\varrho\dot\b}\d+r_{\varrho\dot\b}m'\d\in R^{ln}(V),~~ -(\varrho\dot\b+k_{\varrho\dot\b}\d+r_{\varrho\dot\b}m'\d)\in R^{in}(V)
   & \hbox{if $\pm\varrho\dot\b$ is up-nilpotent hybrid}\\
   \varrho\dot\b+k_{\varrho\dot\b}\d+r_{\varrho\dot\b}m'\d\in R^{ln}(V),~~ -(\varrho\dot\b+k_{\varrho\dot\b}\d+r_{\varrho\dot\b}m'\d)\in R^{ln}(V)
   & \hbox{if $\pm\varrho\dot\b$ is full-locally nilpotent.}
  \end{array}\right. \] 
  This, in particular,  together with (\ref{s_0}) and (\ref{in-supp 1}) and the fact that the level of $V$ is zero, implies  that 
   \begin{align*}
   \upsilon_{s}-i(\varrho\dot\b+k_{\varrho\dot\b}\d+r_{\varrho\dot\b}m'\d)\in \supp(V)\quad (s\in\bbbz^{>0},~1\leq i\leq s),
   \end{align*}
in particular, for \[p_s:=(m_{\varrho s+\varrho s_0}-m_{\varrho s_0})-s(k_{\varrho\dot\b}+r_{\varrho\dot\b}m')\quad\quad(s\in\bbbz^{>0}),\]
we have 
 \begin{align}
 \mu_{\varrho s_0}+s(\varrho\dot\a+k_{\varrho\dot\a}\d)+(p_s-sk_{\varrho\dot\a})\d=&   \mu_{\varrho s_0}+s\varrho\dot\a+p_s\d\nonumber\\
 =& \mu_{\varrho s_0}+s\varrho\dot\a+(m_{\varrho s+\varrho s_0}-m_{\varrho s_0})\d-s(k_{\varrho\dot\b}+r_{\varrho\dot\b}m')\d\label{in-supp}\\
   =&\upsilon_{s}-s(\varrho\dot\b+k_{\varrho\dot\b}\d+r_{\varrho\dot\b}m'\d)
   \in \supp(V)\quad (s\in\bbbz^{>0}).\nonumber
   \end{align}

We next note that for each $s\in\bbbz^{>0},$ using the division algorithm,  we get  $q_s\in\bbbz$ and $0\leq r_s<r_{\varrho\dot\a}$ such that 
\[p_s-sk_{\varrho\dot\a}=q_sr_{\varrho\dot\a}+r_s.\] In particular, (\ref{in-supp}) gives that 
\begin{equation} \label{in-supp V}
\begin{split}
  (\mu_{\varrho s_0}+r_s\d)+s(\varrho\dot\a+k_{\varrho\dot\a}\d)+q_sr_{\varrho\dot\a}\d=
&\mu_{\varrho s_0}+s(\varrho\dot\a+k_{\varrho\dot\a}\d)+(q_sr_{\varrho\dot\a}+r_s)\d
\\
=&\mu_{\varrho s_0}+s(\varrho\dot\a+k_{\varrho\dot\a}\d)+(p_s-sk_{\varrho\dot\a})\d\in\supp(V).
\end{split}
\end{equation}
 Since $\pm\varrho\dot\a$ is full-locally nilpotent,  using    this  together with (\ref{in-supp 1}) and (\ref{s_0}) repeatedly, we get 
\begin{align*}
   \mu_{\varrho s_0}+r_s\d\in \supp(V)\quad (s\in\bbbz^{>0}).
   \end{align*}
Since $r_s$'s are equal  for infinitely many $s$, we get positive integers $s_1,s_2,\ldots$ with
$$r:=r_{s_1}=r_{s_2}=\cdots,$$ and so contemplating (\ref{in-supp V}),  we have 
   \begin{align*}
 (  \mu_{\varrho s_0}+r\d)+s_i(\varrho\dot\a+k_{\varrho\dot\a}\d)+(q_{s_i}r_{\varrho\dot\a})\d\in \supp(V)\quad (i>0).
   \end{align*}
  Using (\ref{in-supp 1}), we get that 
\begin{align*}
 (  \mu_{\varrho s_0}+r\d)+(s_i-1)(\varrho\dot\a+k_{\varrho\dot\a}\d)\in \supp(V)\quad (i>0)
   \end{align*}
which is a contradiction as $\varrho\dot\a+k_{\varrho\dot\a}\d\in  R^{ln}(V) $ while  simple modules have shadow. 
  \end{proof}

Now, we are ready to prove our main result. We recall from  Table \ref{table-ns} that  \begin{equation*}\label{kappa-final}
 \dot R_{ns}\sub (1/\kappa )(\dot R_{re}+\dot R_{re}) \hbox{ in which } \kappa=\left\{
 \begin{array}{ll}
 2 & \hbox{if~}  X=A(2m-1,2n-1)^{(2)},\\
 1& \hbox{otherwise}.
\end{array}
\right. 
\end{equation*}
For a  simple  finite weight $\LL$-module $M$\footnote{We mention that each  simple finite weight $\LL$-module has shadow.}, set  \begin{equation}\label{T1,2}\begin{array}{ll}
T(1)_{re}:=\{\a\in R_{re}\mid \a,-\a\in R_{\rm f-nil}(M)\},\\\\
 T(2)_{re}:=\{\a\in R_{re}\mid \a,-\a\in R_{\rm hyb}(M)\},\\\\
T(i)_{ns}:=R_{ns}\cap (\frac{1}{\kappa}(\dot T(i)_{re}+\dot T(i)_{re})+\bbbz\d),\\\\
T(i):=(\underbrace{T(i)_{ns}\cup T(i)_{re}}_{T(i)^\times})\cup (\underbrace{R_{im}\cap (T(i)^\times+T(i)^\times)}_{T(i)_{im}})& (i=1,2),\end{array}\end{equation}
 as well as
\begin{equation}\label{Tre-ns}T_{re}:=T(1)_{re}\cup T(2)_{re}\andd T:=\underbrace{\bbbz\d}_{T_{im}}\cup T_{re}\cup(\underbrace{R_{ns}\cap (\frac{1}{\kappa} (\dot T_{re}+\dot T_{re})+\bbbz\d)}_{T_{ns}}).\end{equation}
 
 In what follows, by    $\gG_{_S}$ for a {symmetric closed} subsets  $S\sub R,$  we mean 
 \begin{equation*}\label{gs}
 \gG_{_S}=\Bigop{\a\in S}{}\LL^\a.
 \end{equation*}

         \begin{Thm}\label{main} 
Assume  $M$ is a zero level  simple finite weight $\LL$-module. Recall (\ref{T1,2})~$\&$ (\ref{Tre-ns}). If   $T(1)_{re}$ and $T(2)_{re}$  are  nonempty,  we have the following: 
\begin{itemize}
 \item[(i)] $T$ as well as $T(1)$ and $T(2)$ are symmetric closed subsets of $R$. Also, for $i=1,2,$ $T(i)_{im}$ is either $\bbbz\d$ or $2\bbbz\d.$ Moreover, if   $T(2)_{im}=2\bbbz\d$, then, $T(2)\sub R_0.$
\item[(ii)]   Assume $V$ is a simple $\mathcal{G}_{_T}$-submodule of $M$ and recall notations used in \S~\ref{par-subsets}.
\begin{itemize}
\item[(a)]
There exist a linear functional $\bs{\zeta_1}:\sspan_\bbbr T\longrightarrow \bbbr$ with $ {\rm ker}(\bs\zeta_1)=\bbbz\d\cup T(2)$ and a simple $ \gG_{_{T_{{\bs{\zeta_1}}}^\circ}}$-module $U$ such that 
\[M={\rm Ind}^{\gG_{_T}}_{\bs\zeta_1}(U).\] 
\item[(b)]  If $T(2)_{im}=\bbbz\d,$ there is a linear functional $\bs\zeta_2:\sspan_\bbbr T_{{\bs{\zeta_1}}}^\circ\longrightarrow \bbbr$ and a simple 
$\gG_{_{(T_{{\bs{\zeta_1}}}^\circ)_{\bs{\zeta_2}}^\circ}}$-module $N$ such that 
\[U={\rm Ind}_{\bs\zeta_2}^{\gG_{_{T_{\bs\zeta_1}^\circ}}}(N).\] 
\item[(c)]  If $T(2)_{im}=2\bbbz\d$ and $V$ has bounded weight multiplicities, then,  the intersection $\sum_{k\in\bbbz}\LL^{k\d}$ with the  center $Z(\gG_{_{T_{{\bs{\zeta_1}}}^\circ}})$ of $\gG_{_{T_{{\bs{\zeta_1}}}^\circ}}$ acts trivially on $U$ and there is a linear functional $\bs\zeta_2:\sspan_\bbbr T_{{\bs{\zeta_1}}}^\circ\longrightarrow \bbbr$ and a simple $\gG_{_{(T_{{\bs{\zeta_1}}}^\circ)_{\bs{\zeta_2}}^\circ}}$-module $N$ such that 
\[U={\rm Ind}_{\bs\zeta_2}^{\gG_{_{T_{\bs\zeta_1}^\circ}}}(N).\] 
\end{itemize}
\item[(iii)]  Suppose that  $V$ is a simple $\mathcal{G}_{_T}$-submodule of $M$. Assume  $V$ has bounded weight multiplicities if $T(2)_{im}=2\bbbz\d$. Then, there are  linear functionals  \[\bs{\zeta_1}:\sspan_\bbbr T\longrightarrow \bbbr,~\bs\zeta_2:\sspan_\bbbr T_{{\bs{\zeta_1}}}^\circ\longrightarrow \bbbr\andd \bs\zeta_3:\sspan_\bbbr \underbrace{(T_{{\bs{\zeta_1}}}^\circ)^\circ_{\bs\zeta_2}}_{T'}\longrightarrow \bbbr\] 
such that $T'=(T^\circ_{_{\bs\zeta_1}})_{_{\bs\zeta_2}}^\circ$ is finite, and   a  cuspidal module $\Omega$ over the finite dimensional superalgebra $\gG_{_{T'}}=\Bigop{\a\in T'}{}\LL^\a$
such that 
  $M={\rm Ind}_{_{\bs\zeta_1,\bs\zeta_2}}^{\gG_{_T}}{\rm Ind}_{_{\bs\zeta_3}}^{\gG_{_{T'}}}(\Omega).$ 
\end{itemize}
\end{Thm}
    \begin{rem}
\rm{Recalling  (\ref{r12}),  by setting  $T(i):=R(i)$ ($i=1,2$) in Theorem~\ref{main},
we get  a complete characterization of quasi-integrable $\LL$-modules, at the critical level (i.e., the zero level).}    \end{rem} 
\begin{proof}[{{Proof of Theorem~\ref{main}}}]
(i)   is easy to see using the fact that for $\a,\b\in R_{re}$ with $\a+\b\in R_{re},$ if $\a,\b\in R^{ln}$ (resp. $\a,\b\in R^{in}$), then  $\a+\b\in R^{ln}$ (resp. $\a+\b\in R^{in}$); see \cite[Theorem 4.7, Lemma 3.5]{you13}. For the last assertion, if $T(2)_{im}=2\bbbz\d,$  $\dot T(2)$ does not contain nonzero elements $\dot\a$ with $\dot\a+\bbbz\d\sub T(2)$. Therefore,  $T(2)\sub R_0;$ see Table~\ref{table1}.

(ii)(a) Set $Q:=R_0\cap T(1).$
Then,  \[\dot Q_0=\{\dot\a\in \dot R\mid Q\cap (\dot\a+\bbbz\d)\neq \emptyset\}\] is a finite root system. Fix a base $\D$ of $\dot Q_0$ 
and  define a linear functional  $\bs{\zeta_1}$ on $\sspan_\bbbr \dot T(1)_{re}=\sspan_\bbbr \dot Q_0$ with $\bs{\zeta_1}(\D)\sub\bbbr^{> 0}$ and $\bs\zeta_1(\dot \a)\neq 0$ for all $\dot\a\in \dot T(1)_{ns}$\footnote{We note that $\dot Q_0$ is a direct sum of orthogonal finite root systems, in particular, if $(\dot T(1))_{ns}\ne \emptyset,$ each nonsingular root of $\dot T(1)$ is a sum of two elements of real spans of two different irreducible components of $\dot Q_0.$}. Extend $\bs{\zeta_1}$ to a linear functional on $\sspan_\bbbr T$ with 
\begin{equation*}\label{extend-T}\bs{\zeta_1}(\a)=0\quad(\forall \a\in T(2)\cup\bbbz\d).
\end{equation*}
We have $T=T_{\bs{\zeta_1}}^+\uplus T_{\bs{\zeta_1}}^\circ\uplus T_{\bs{\zeta_1}}^-$ with
\begin{align*}\label{T-zeta-1}
T_{\bs{\zeta_1}}^{\pm}=&\{\dot\b+r\d\in T(1)\mid \bs{\zeta_1}(\dot\b)\gtrless0\}\\
\cup& \{\frac{1}{\kappa}(\dot\b+\dot\gamma)+r\d\in T\mid  \dot\b\in \dot T(1)_{re}, \dot\gamma\in \dot T(2)_{re},r\in\bbbz,\bs{\zeta_1}(\dot\b)\gtrless0\}.\nonumber
\end{align*}
We have $T^\circ_{\bs{\zeta_1}}=T(2)\cup\bbbz\d$ and shall show  $U:=V^{(\gG_{_T})^+_{\bs{\zeta_1}}}=\{v\in V\mid \LL^\a v=\{0\}~~(\forall \a\in T^+_{_{\bs\zeta_1}})\}$ is a nonzero module over 
\[\gG_{T^\circ_{\bs\zeta_1}}=\Bigop{\a\in \bbbz\d\cup T(2)}{}\LL^\a=\sum_{k\in \bbbz}\LL^{(2k+1)\d}+\gG_{_{T(2)}}.\]

We just need to show that $V^{(\gG_{_T})^+_{\bs{\zeta_1}}}\neq \{0\}.$
Consider  $V$ as a module over  $\gG_{_{T_1}}=\Bigop{\a\in T_1}{}\LL^\a$ and use  Proposition~\ref{1}(c) to get that the set 
\begin{equation}\label{w}
	\begin{split}
W:=&\{v\in V\mid \LL^{\dot \a+r\d} v=\{0\}\quad(\dot\a\in \dot Q_0^+(\D),~r\in\bbbz, \dot\a+r\d\in R_0)\}\\
=&\{v\in V\mid \LL^{\dot \a+r\d} v=\{0\}\quad(\dot\a\in \dot T(1)_{re},~r\in\bbbz, \dot\a+r\d\in R_0\cap T(1),~\bs{\zeta_1}(\dot\a)>0)\}
\end{split}
\end{equation}
is a nonzero vector space. For each nonzero weight vector $u\in W,$ set
\[\aa_u:=\{\dot\a\in \dot T\setminus \dot T(2)\mid  \bs{\zeta_1}(\dot\a)>0,\exists r\in \bbbz\hbox{ s.t. }~\LL^{\dot\a+r\d}u\neq \{0\}\} .\]  We note that 
\begin{equation}\label{in R1}
(\aa_u+\bbbz\d)\cap R\sub R_1.
\end{equation}
To complete the proof, we need to show that there is $u_0\in W$ with $\aa_{u_0}=\emptyset.$ To this end, assume 
$u_0$ is a nonzero weight vector in $W$ (see (\ref{w})) such that $\aa_{u_0}$ is of minimum cardinality.  We claim that $\aa_{u_0}=\emptyset.$  If to the contrary,    $\aa_{u_0}\not=\emptyset,$ we fix    $\dot\a_*\in\aa_{u_0}$  such that   \[\bs{\zeta}_1(\dot\a_*)=
   {\rm max}\{\bs{\zeta}_1(\dot\a)\mid \dot\a\in \aa_{u_0}\}\] as well as  $m_0\in\bbbz$ with $\LL^{\dot\a_*+m_0\d}u_0\neq \{0\}$  and $0\neq u_1\in \LL^{\dot\a_*+m_0\d}u_0$. For $\dot\a\in \dot T$  with   $\bs{\zeta_1}(\dot\a)>0$ and  $r\in\bbbz$, we have
 \begin{align*}
 \LL^{\dot\a+r\d}u_1\sub \LL^{\dot\a+r\d}\LL^{\dot\a_*+m_0\d}u_0\sub& [\LL^{\dot\a+r\d},\LL^{\dot\a_*+m_0\d}]u_0+\LL^{\dot\a_*+m_0\d}\LL^{\dot\a+r\d}u_0\\
\sub&  \LL^{\dot\a+\dot\a_*+(r+m_0)\d}u_0+\LL^{\dot\a_*+m_0\d}\LL^{\dot\a+r\d}u_0= \LL^{\dot\a_*+m_0\d}\LL^{\dot\a+r\d}u_0;
\end{align*}  
the last equality is due to the choice of $\dot\a_*.$ This  easily implies that 
   \begin{itemize}
   \item[(1)]
   $u_1\in W$,
   \item[(2)] $\aa_{u_1}=\aa_{u_0}.$  
   \end{itemize}
Therefore, $ \LL^{\dot\a_*+m_1\d}u_1\neq \{0\}$ for some $m_1\in\bbbz;$ we pick $0\neq u_2\in \LL^{\dot\a_*+m_1\d}u_1$. Continuing this process, we get 
$m_i\in\bbbz$  and  nonzero weight vectors $u_i$ $(i\geq 0)$ such that \[\dot\a_{*}+m_0\d,\dot\a_{*}+m_i\d\in R_1\andd 0\neq u_i\in \LL^{\dot\a_{*}+m_{i-1}\d}u_{i-1}\quad (i\geq 1).\]
Assume  $\mu_0$ is  the weight of $u_0$. We have 
 \begin{equation*}\label{mur}
 \mu_r:=\mu_0+r\dot\a_*+(m_0+\cdots+m_{r-1})\d\in \supp(V)\quad\quad(r\in\bbbz^{>0}).
 \end{equation*}
 By Lemma~\ref{M--30},
 $\dot\a_*$ is a real root. So, we have $\dot\a_*\in \dot T(1)_{re}$.
 Pick $r_0\in\bbbz^{>0}$ with 
 \[2(\mu_{r_0},\dot\a_*)/(\dot\a_*,\dot\a_*)=2(\mu_0+r_0\dot\a_*,\dot\a_*)/(\dot\a_*,\dot\a_*)>0.\] We have 
 \[\mu_{r_0+s}=\mu_{r_0}+s\dot\a_*+(m_{r_0}+\cdots+m_{r_0+s-1})\d\in\supp(V)\quad (s\in\bbbz^{>0}).\] 
Contemplating (\ref{in R1}), for each $i\geq 0$, $\dot\a_*+m_i\d\in R_1$. So, by (\ref{exp})(2), 
for each odd positive integer $s$, we have 
 $\dot\a_*+(-(s-2)m_{r_0}+m_{r_0+1}+\cdots+m_{r_0+s-1})\d\in T(1)$ and so,  we get using (\ref{in-supp 1}) and the fact that the level of $M$ is zero that 
\begin{align*}
 &\mu_{r_0}+(s-1)(\dot\a_*+m_{r_0}\d)\\
 =~&\mu_{r_0}+s\dot\a_*+(m_{r_0}+\cdots+m_{r_0+s-1})\d-(\dot\a_*+(-(s-2)m_{r_0}+m_{r_0+1}+\cdots+m_{r_0+s-1})\d)\\
 \in~ &\supp(V)\quad\quad (s\in 2\bbbz^{\geq 0}+1).
\end{align*} 
 This contradicts the fact that  $\dot\a_*+m_{r_0}\d\in T(1)_{re}\subseteq R^{ln}(V)$; see (\ref{BC}). Therefore, $\aa_{u_0}=\emptyset$ and so 
 \begin{equation}\label{U-MM}
 \parbox{5.5in}{$U=V^{(\gG_{_T})^+_{\bs{\zeta_1}}}=\{v\in V\mid \LL^\a v=\{0\}~~(\forall \a\in T^+_{_{\bs\zeta_1}})\}$ is a nonzero $\gG_{T^\circ_{\bs\zeta_1}}$-module.}
 \end{equation}
So, by Proposition~\ref{ind}, $U$ is a simple module over $\gG_{T^\circ_{\bs\zeta_1}}$ and $M=\rm Ind_{\bs\zeta_1}^{\gG_{_T}}(U).$
\smallskip

(ii) Using Lemma~\ref{lem:up and down},  either all real roots of $T(2)$ are up-nilpotent hybrid or all real roots of $T(2)$ are down-nilpotent hybrid. Assume $\sg=1$ if  all real roots of $T(2)$ are up-nilpotent hybrid and assume $\sg=-1$ otherwise.
 One sees that 
\[ P_2:=T(2)_{re}^{ln}\cup -T(2)_{re}^{in}\cup  (\sg\bbbz^{\geq0}\d\cap T(2))\]
satisfies 
\[(P_2+P_2)\cap (T(2)\setminus R_{ns})\sub P_2\andd P_2\cup-P_2= T(2)\setminus R_{ns},\] that is,
$P_2$ is a parabolic subset of $T(2)\setminus R_{ns}.$
Moreover, by Lemma~\ref{lem:exist functional}, there is a linear functional $\bs{\zeta_2}$ on $\sspan_\bbbr (T(2)\setminus R_{ns})=\sspan_\bbbr T(2)$ such that 
\[P_2=\{\a\in T(2)\setminus R_{ns}\mid \bs{\zeta_2}(\a)\geq 0\}.\]
By Proposition~\ref{1}(b), there is  $K\sub T(2)\cap R_0$ such that 
$$\ft:=\sum_{\a\in \{0\}\cup K^\times}\LL^\a+\sum_{\a,\b\in  K^\times}[\LL^\a,\LL^\b]$$
is an affine Lie algebra, up to a central space.  We have  $K=K^+\cup K^\circ\cup K^-$ in which 
\[K^\circ:=\{\a\in K\mid \bs{\zeta_2}(\a)=0\}\andd  K^\pm:=\{\a\in K\mid \bs{\zeta_2}(\a)\gtrless0\}.\] 
Moreover, $U$ is a finite weight $\ft$-module and
\[K^+\cap K_{re}\sub K^{ln}(U)\andd K^-\cap K_{re}\sub K^{in}(U).\] This implies that there is $\lam\in \supp(U)$  and a positive integer $p$ such that 
$\lam+p\bbbz^{>0}\d\cap \supp(U)=\emptyset$; see \cite[Lemma 5.1]{you8}. Using the same argument as in \cite[Theorem 5.8]{you8} for the finite weight module $U$, as a module over  $\gG_{T(2)},$ we have \begin{equation}\label{udot}
N:=\{v\in U\mid \LL^\a v=\{0\}  (\a\in T(2), \bs\zeta_2(\a)>0)\}
\end{equation} is nonzero. Now, we can prove (ii)(b),(c):

(ii)(b) Since $T(2)_{im}=\bbbz\d,$ we have
\begin{equation*}
N=\{v\in U\mid \LL^\a v=\{0\}~~(\a\in T(2), \bs\zeta_2(\a)>0)\}=\{v\in U\mid \LL^\a v=\{0\} ~~(\a\in T^\circ_{\bs\zeta_1}, \bs\zeta_2(\a)>0)\}
\end{equation*} and so, using Proposition~\ref{ind}, we have $U={\rm Ind}_{\bs\zeta_2}^{\gG_{_{T_{\bs\zeta_1}^\circ}}}(N).$

(ii)(c)  Recall  Remark~\ref{rem-MMM}, specially $\mathscr{K}$, $\mathcal{I}$ and $x_{4k+2}, y_{4k+2}$'s ($k\in\bbbz$). Set
\begin{equation*}\label{e1}
\sce:=\left\{
\begin{array}{ll}
\Bigop{k\in\bbbz}{}\LL^{(2k+1)\d}& X\neq A(2m,2n)^{(4)},\\
\Bigop{k\in\bbbz}{}\LL^{(2k+1)\d}\op\Bigop{k\in\bbbz}{}\bbbc x_{4k+2}\op\Bigop{k\in\bbbz}{}\bbbc y_{4k+2} & X= A(2m,2n)^{(4)},
\end{array}
\right.
\end{equation*}
and 
\begin{equation*}\label{f1}
\scf:=\left\{
\begin{array}{ll}
\fh\op\bbbc c\op\sce\op\displaystyle{\sum_{\a\in T(2)^\times}\LL^\a}\op\Bigop{0\neq k\in\bbbz}{}\LL^{2k\d}& X\neq A(2m,2n)^{(4)},\\
\fh\op\bbbc c\op\sce\op \displaystyle{\sum_{\a\in T(2)^\times}\LL^\a}\op\mathscr{K}\op((1-\d_{m,n})\ii)\ot t^2\bbbc[t^{\pm2}])& X= A(2m,2n)^{(4)}.
\end{array}
\right.
\end{equation*}
We have 
\[\scf_{cc}\rtimes \bbbc d=\gG_{T^\circ_{\bs\zeta_1}}/\bbbc c.\]
Since the level of $V$ is zero, the simple $\gG_{T^\circ_{\bs\zeta_1}}$-module  $U$ (see (\ref{U-MM})) is a simple module over $\gG_{T^\circ_{\bs\zeta_1}}/\bbbc c=\scf_{cc}\rtimes \bbbc d$ and so Proposition~\ref{tight} implies that 
\begin{equation}\label{zc}
\parbox{2.6in}{
$Z(\scf_{cc})\cap \Bigop{0\neq k\in \bbbz}{}\LL^{k\d}$ acts trivially on $U$.}
\end{equation}
Now we address the following  cases $X\neq A(2m,2n)^{(4)}$ and $X= A(2m,2n)^{(4)}$ separately:

\noindent $\bullet$ $X\neq A(2m,2n)^{(4)}.$ Contemplating (\ref{zc}), we have $\Bigop{k\in\bbbz}{}\LL^{(2k+1)\d} U=\{0\}$
and so recalling (\ref{udot}), we have \begin{equation*}
N=\{v\in U\mid \LL^\a v=\{0\}~~(\a\in T(2), \bs\zeta_2(\a)>0)\}=\{v\in U\mid \LL^\a v=\{0\} ~~(\a\in T^\circ_{\bs\zeta_1}, \bs\zeta_2(\a)>0)\}.
\end{equation*} Therefore, using Proposition~\ref{ind}, $N$ is a simple $\gG_{(T_{\bs{\zeta}_1}^\circ)^\circ_{\bs\zeta_2}}$-module and that  $U={\rm Ind}_{\bs\zeta_2}^{\gG_{T^\circ_{\bs\zeta_1}}}(N).$

\noindent $\bullet$ $X=A(2m,2n)^{(4)}.$ Recalling (\ref{zc}) and (\ref{y's}), we have $\Bigop{k\in\bbbz}{}\bbbc y_{4k+2}U=\{0\}$. We also have $\Bigop{k\in\bbbz}{}\LL^{(2k+1)\d}U\sub U.$ 
So, using  the same argument as in {\cite[Proposition~3.2]{KY}}, we get that 
\[U':=\{u\in U\mid \LL^{(2k+1)\d}u=0\quad(k\in\bbbz)\}\] is nonzero. We have that $U'$ is a $\gG_{_{T(2)}}$-module and as above we get that 
{\small \[N'=\{v\in U'\mid \LL^\a v=\{0\}~~(\a\in T(2), \bs\zeta_2(\a)>0)\}\sub \{v\in U\mid \LL^\a v=\{0\} ~~(\a\in T^\circ_{\bs\zeta_1}, \bs\zeta_2(\a)>0)\}=N\]} is nonzero.
 Therefore, using Proposition~\ref{ind}, $N$ is a simple $\gG_{(T_{\bs{\zeta}_1}^\circ)^\circ_{\bs\zeta_2}}$-module and that  $U={\rm Ind}_{\bs\zeta_2}^{\gG_{T^\circ_{\bs\zeta_1}}}(N).$

(iii) Since  $(T^\circ_{_{\bs\zeta_1}})_{_{\bs\zeta_2}}^\circ$ is finite, the result follows from \cite{D-M-P}.
 \end{proof}

\end{document}